\documentclass[11pt]{article}
\usepackage{srcltx}
\usepackage{eurosym}
\usepackage{mathtools}
\usepackage{amsmath}
\usepackage{amsfonts}
\usepackage{amssymb}
\usepackage{amsthm}
\usepackage{bm}
\usepackage{graphicx}
\usepackage{mathrsfs}
\usepackage{xcolor}
\allowdisplaybreaks
\usepackage{exscale}
\usepackage{color, soul}
\usepackage{latexsym}
\usepackage{authblk}
\usepackage{comment}
\allowdisplaybreaks

\usepackage[colorlinks,plainpages=true,pdfpagelabels,hypertexnames=true,colorlinks=true,pdfstartview=FitV,linkcolor=blue,citecolor=red,urlcolor=black]{hyperref}
\PassOptionsToPackage{unicode}{hyperref}
\PassOptionsToPackage{naturalnames}{hyperref}
\usepackage{enumerate}
\usepackage[shortlabels]{enumitem}
\usepackage{bookmark}
\usepackage{wasysym}
\usepackage{esint}
\usepackage[ddmmyyyy]{datetime}
\usepackage[margin=1in]{geometry}
\makeatletter
\g@addto@macro\normalsize{%
	\setlength\abovedisplayskip{4pt}
	\setlength\belowdisplayskip{4pt}
	\setlength\abovedisplayshortskip{4pt}
	\setlength\belowdisplayshortskip{4pt}
}
\numberwithin{equation}{section}
\everymath{\displaystyle}
\usepackage[capitalize,nameinlink]{cleveref}
\crefname{section}{Section}{Sections}
\crefname{subsection}{Subsection}{Subsections}
\crefname{condition}{Condition}{Conditions}
\crefname{hypothesis}{Hypothesis}{Conditions}
\crefname{assumption}{Assumption}{Assumptions}
\crefname{lemma}{Lemma}{Lemmas}
\crefname{definition}{Definition}{Definitions}
\DeclareMathOperator{\sgn}{sgn}

\crefformat{equation}{\textup{#2(#1)#3}}
\crefrangeformat{equation}{\textup{#3(#1)#4--#5(#2)#6}}
\crefmultiformat{equation}{\textup{#2(#1)#3}}{ and \textup{#2(#1)#3}}
{, \textup{#2(#1)#3}}{, and \textup{#2(#1)#3}}
\crefrangemultiformat{equation}{\textup{#3(#1)#4--#5(#2)#6}}%
{ and \textup{#3(#1)#4--#5(#2)#6}}{, \textup{#3(#1)#4--#5(#2)#6}}%
{, and \textup{#3(#1)#4--#5(#2)#6}}

\Crefformat{equation}{#2Equation~\textup{(#1)}#3}
\Crefrangeformat{equation}{Equations~\textup{#3(#1)#4--#5(#2)#6}}
\Crefmultiformat{equation}{Equations~\textup{#2(#1)#3}}{ and \textup{#2(#1)#3}}
{, \textup{#2(#1)#3}}{, and \textup{#2(#1)#3}}
\Crefrangemultiformat{equation}{Equations~\textup{#3(#1)#4--#5(#2)#6}}%
{ and \textup{#3(#1)#4--#5(#2)#6}}{, \textup{#3(#1)#4--#5(#2)#6}}%
{, and \textup{#3(#1)#4--#5(#2)#6}}

\crefdefaultlabelformat{#2\textup{#1}#3}

\newtheorem{theorem} {Theorem}[section]
\newtheorem{proposition} [theorem]{Proposition}

\newtheorem{lemma}[theorem]{Lemma}

\newtheorem{counter example}[theorem]{Counter Example}
\newtheorem{remark}[theorem] {Remark}
\newtheorem{definition}[theorem] {Definition}

\newtheorem{assumption}[theorem]{Assumption}

\def\CC{{\rm \kern.24em \vrule width.02em height1.4ex depth-.05ex \kern-.26emC}}

\def\TagOnRight

\def\AA{{it I} \hskip-3pt{\tt A}}

\def\QQ{\rlap {\raise 0.4ex \hbox{$\scriptscriptstyle |$}} {\hskip -0.1em Q}}

\makeatletter
\newcommand{\vo}{\vec{o}\@ifnextchar{^}{\,}{}}
\makeatother

\def\YYint#1#2#3{{\setbox0=\hbox{$#1{#2#3}{\iint}$}
		\vcenter{\hbox{$#2#3$}}\kern-.50\wd0}}

\def\XXint#1#2#3{{\setbox0=\hbox{$#1{#2#3}{\int}$}
		\vcenter{\hbox{$#2#3$}}\kern-.50\wd0}}

\makeatletter
\def\namedlabel#1#2{\begingroup
	\def\@currentlabel{#2}%
	\label{#1}\endgroup
}
\makeatother
\makeatletter
\newcommand{\rmh}[1]{\mathpalette{\raisem@th{#1}}}
\newcommand{\raisem@th}[3]{\hspace*{-1pt}\raisebox{#1}{$#2#3$}}
\makeatother

\newcounter{desccount}

\newcommand{\descref}[2]{\hyperref[#1]{\textnormal{\textcolor{black}{}\textcolor{blue}{ #2}\textcolor{black}{}}}}

\newcommand{\dref}[2]{\hyperref[#1]{\textcolor{black}{(}\textcolor{blue}{\bf #2}\textcolor{black}{)}}}
\newcommand{\be} {\begin{eqnarray}}
	\newcommand{\ee} {\end{eqnarray}}
\newcommand{\Bea} {\begin{eqnarray*}}
	\newcommand{\Eea} {\end{eqnarray*}}

\newcommand{\eps} {\epsilon}

\newcounter{whitney}
\refstepcounter{whitney}

\newcounter{ineqcounter}
\refstepcounter{ineqcounter}
\makeatletter
\def\ps@pprintTitle{%
	\let\@oddhead\@empty
	\let\@evenhead\@empty
	\def\@oddfoot{}%
	\let\@evenfoot\@oddfoot}
\makeatother
\usepackage[doublespacing]{setspace}
\usepackage[titletoc,toc,page]{appendix}

\makeatletter
\newcommand{\refcheckize}[1]{%
	\expandafter\let\csname @@\string#1\endcsname#1%
	\expandafter\DeclareRobustCommand\csname relax\string#1\endcsname[1]{%
		\csname @@\string#1\endcsname{##1}\wrtusdrf{##1}}%
	\expandafter\let\expandafter#1\csname relax\string#1\endcsname
}
\makeatother

\refcheckize{\cref}
\refcheckize{\Cref}


\makeatletter
\newcommand{\mainsectionstyle}{%
	\renewcommand{\@secnumfont}{\bfseries}
	\renewcommand\section{\@startsection{section}{2}%
		\z@{.5\linespacing\@plus.7\linespacing}{-.5em}%
		{\normalfont\bfseries}}%
}
\makeatother
\usepackage{pgf,tikz}
\usetikzlibrary{arrows}
\usetikzlibrary{decorations.pathreplacing}
\usepackage{enumitem}    

\usepackage{xpatch}
\xpatchcmd{\MaketitleBox}{\hrule}{}{}{}
\xpatchcmd{\MaketitleBox}{\hrule}{}{}{}

\date{}

\usepackage{scalerel}

\makeatletter

\usepackage{subcaption}
\usepackage{hyperref}
\usepackage{caption}
\usepackage{subcaption}

\usepackage[utf8]{inputenc}
\allowdisplaybreaks

\DeclareCaptionSubType*[Alph]{figure}
\title{Well-posedness for a class of $n\times n$ hyperbolic systems and shock stability for the viscous approximation}
\author[1]{Rahul Barthwal\thanks{\href{mailto:rahul.barthwal@mathematik.uni-stuttgart.de}{rahul.barthwal@mathematik.uni-stuttgart.de}}}
\author[2]{Lilu Sahu\thanks{\href{mailto:lilusahu0123@kgpian.iitkgp.ac.in}{lilusahu0123@kgpian.iitkgp.ac.in}, 
}}
\affil[1]{\footnotesize Institute of Applied Analysis and Numerical Simulation, University of Stuttgart\\

Pfaffenwaldring 57, 70569 Stuttgart, Germany}
\affil[2]{\footnotesize Department of Mathematics, Indian Institute of Technology Kharagpur\\

Kharagpur, 721302 West Bengal, India}

\begin{document}

\maketitle
\begin{abstract}
We establish the global well-posedness of weak solutions to the Cauchy problem for a broad class of multi-dimensional $n\times n$ Keyfitz-Kranzer type systems with homogeneous flux. Our proof relies on a novel viscous approximation that yields the necessary compactness estimates. By exploiting the specific structure of this proposed approximation, we derive a priori uniform $L^{\infty}$ bounds and prove the $L^1_{loc}$ precompactness of the sequence of approximate solutions. This framework allows us to rigorously justify the vanishing viscosity limit. Furthermore, under suitable assumptions on the initial data, we employ the relative entropy method to analyze the $L^2$-time decay of large perturbations of the viscous shock, up to a dynamical shift.
\end{abstract}
\medskip \\
{\textbf{MSC codes.}  35L40 35L45 35L65 35A01  35B40 


\section{Introduction}
Let $T>0$, $0<d\leq 3$, $n\in\mathbb N$, and consider the Cauchy problem for the system
\begin{equation}\label{eq:homogeneous_KK}
\begin{array}{rcll}
\mathbf U_t+\sum_{j=1}^d  \partial_{x_j}\big(\mathbf U V(\mathbf U)\big)=0\quad \text{in }\Omega_T:=\mathbb R^d\times(0,T),\\
\end{array}
\end{equation}
with initial data
\begin{equation}\label{initial_Data}
    \mathbf U(\cdot,0)=\mathbf U_0\quad \text{in }\mathbb R^d.
\end{equation}
In \eqref{eq:homogeneous_KK}, $\mathbf U=(u_1,\ldots,u_n)^\top:
\Omega_T\rightarrow\mathcal U\subset\mathbb R^n,$ is the unknown conserved vector 
defined on a state space $\mathcal{U}$. To make the discussion in the following precise, we make the following assumption on the state space $\mathcal{U}$ and the function $V(\mathbf{U})$.
\begin{assumption}[State space and homogeneous velocity]
\label{assumption_state_space}
Let $\mathcal{G} \subset \mathbb{R}^{n-1}$ be a nonempty, open, and convex set. Define the open convex cone $\mathcal{C}$ as
\[
\mathcal{C} := \left\{ \mathbf{U} = u_1(1, \boldsymbol\theta)^\top \in \mathbb{R}^n : u_1 > 0, \, \boldsymbol\theta \in \mathcal{G} \right\}\subset \mathbb{R}^n,
\]
where $ \boldsymbol\theta :=(\theta_2,\ldots, \theta_n)= \left(\frac{u_2}{u_1}, \dots, \frac{u_n}{u_1}\right).$

Let $V \in C^\infty(\mathcal{C}; (0, \infty))$ be a homogeneous function of degree $k > 0$, such that for all $\lambda > 0$ and $\mathbf{U} \in \mathcal{C}$,
\begin{equation}\label{eq:velocity_homogeneity}
V(\lambda \mathbf{U}) = \lambda^k V(\mathbf{U}).
\end{equation}
Equivalently, defining $g \in C^\infty(\mathcal{G}; (0, \infty))$ by $g(\boldsymbol\theta) := V(1, \boldsymbol\theta)$, we have
\[
V(\mathbf{U}) = u_1^k g(\boldsymbol\theta). \qquad
\]
We restrict our analysis to a uniformly positive state space $\mathcal{U}$. Let $c > 0$ be a constant and $\mathcal{K} \Subset \mathcal{G}$ be a compact subset. We define our state space $\mathcal{U} \subset \mathcal{C}$ as
\[
\mathcal{U} := \left\{ \mathbf{U} = u_1(1, \boldsymbol\theta)^\top \in \mathcal{C} : \,u_1 \geq c, \, \boldsymbol\theta \in \mathcal{K} \right\}.
\]
Since $u_1 \geq c > 0$ and $g$ attains a strictly positive lower bound on the compact set $\mathcal{K}$, the velocity $V$ is uniformly positive on $\mathcal{U}$ such that there exists a constant $V_0 > 0$ with $V(\mathbf{U}) \geq V_0$ for all $\mathbf{U} \in \mathcal{U}$.
\end{assumption}
In view of Assumption \ref{assumption_state_space}, the variables $V$ and $\boldsymbol\theta$ determine the original state $\mathbf{U}$ uniquely. Precisely, $\mathbf{U}=(u_1, u_2, \ldots, u_n)^\top$ is recovered as
\begin{equation}\label{eq:reconstruction_formula}
    u_1
    =\left(\frac{V}{g(\boldsymbol\theta)}\right)^{1/k},
    \qquad
    u_i=\theta_i u_1,
    \quad i=2,\ldots,n.
\end{equation}
Moreover, the homogeneity of $V$ allows us to apply Euler's identity, which gives
\begin{equation}\label{eq:Euler_homogeneous}
    \nabla_{\mathbf U}V(\mathbf U)\cdot\mathbf U
    =kV(\mathbf U).
\end{equation}
This particular identity plays a crucial role in obtaining a precise mathematical structure (see \eqref{eq:V_epsilon_equation}-\eqref{eq:xi_epsilon_equation} below). The mathematical structure \eqref{eq:V_epsilon_equation}-\eqref{eq:xi_epsilon_equation}, transformation \eqref{eq:reconstruction_formula} and the Euler identity \eqref{eq:Euler_homogeneous}, are the basic mechanisms behind the compactness
argument utilized in this article to obtain the global well-posedness of the weak solutions for the Cauchy problem \eqref{eq:homogeneous_KK}-\eqref{initial_Data}.

It is noteworthy that several mathematical
models of physical significance take the form \eqref{eq:homogeneous_KK}. In particular, for $n=1$, homogeneity implies
$V(u)=V(1)u^k$ for $u>0$, and hence, after normalization, the
system \eqref{eq:homogeneous_KK} reduces to a scalar conservation law of the form
\[
u_t+\sum_{j=1}^d\partial_{x_j}(u^{k+1})=0.
\]
In particular, for $n=1$ and $k=1$, \eqref{eq:homogeneous_KK} reduces to the multi-dimensional Burgers equation.
The choice
$
V(\mathbf U)=|\mathbf U|^k
$
produces the homogeneous subclass of the symmetric Keyfitz--Kranzer
systems, which is related to the models arising in elasticity theory and simplified
magnetohydrodynamics; see, e.g., \cite{freistuhler1991rotational, freistuhler1994cauchy, keyfitz1980system}. For two-component systems with $V=u_1$, system \eqref{eq:homogeneous_KK} becomes the simplest gas-dynamic type model \cite{shen2011global, tan1994delta}. It is important to point out that the class \eqref{eq:homogeneous_KK} also contains nonsymmetric hyperbolic systems. For example, system \eqref{eq:homogeneous_KK} with the two-component velocities $V(u_1,u_2)=\frac{u_1u_2}{2}$ or $V(u_1,u_2)=\frac{u_1u_2}{2}+\frac{u_1^2}{3}$ describes the dynamics of first-order models for thin
film flows with soluble anti-surfactant transport and gravity; see, e.g., 
\cite{barthwal2026global, barthwal2023construction, barthwal2025hyperbolic, barthwal2025existence}. In both of these cases, $V$ is a homogeneous function of degree $2$ and thus belongs to the class of hyperbolic systems \eqref{eq:homogeneous_KK}. Thus, system \eqref{eq:homogeneous_KK} provides a unified framework that includes several important classes of problems  that arise in continuum mechanics.

A classical way to construct weak solutions of a system of the form
\begin{equation}\label{general_system}
    \mathbf U_t
    +\sum_{j=1}^d\partial_{x_j}\mathbf F_j(\mathbf U)
    =0
\end{equation}
is to use a viscous approximation of the form
\begin{equation}\label{general_viscous}
    \mathbf U_t^\epsilon
    +\sum_{j=1}^d\partial_{x_j}
       \mathbf F_j(\mathbf U^\epsilon)
    =
    \epsilon\sum_{j=1}^d\partial_{x_j}
    \left(
        \mathbf B(\mathbf U^\epsilon)
        \partial_{x_j}\mathbf U^\epsilon
    \right).
\end{equation}
Here, $\epsilon>0$ and
$\mathbf B:\mathcal U\to\mathbb R^{n\times n}$ is a suitably chosen diffusion
matrix. In several important classes of hyperbolic systems of the form \eqref{general_system}, the standard choice
$\mathbf B=\mathbf I$ provides the estimates required to justify the
vanishing-viscosity limit in the system \eqref{general_viscous}; see, e.g.,
\cite{bianchini_annals,diperna_cc,heibig_1994,Lu_SIMA}. However, for more
delicate nonlinear systems, the diffusion matrix $\mathbf{B}$ must respect the
structure of the convective part of the hyperbolic system. A notable example is the work of
Chen \& Perepelitsa \cite{chen2015vanishing}, where carefully designed
viscous terms are used to construct global finite-entropy
solutions of the compressible Euler equations with spherical symmetry.

By far, global existence of weak (entropy) solutions is well understood for several symmetric Keyfitz--Kranzer systems, particularly when the velocity depends only on
the modulus of the state; see, e.g.,
\cite{ambrosio2005well, lu2002hyperbolic} and references cited therein. However, these results
do not directly provide a vanishing-viscosity construction for a general
positive homogeneous velocity $V(\mathbf U)$, whose level sets need
not be radial. The underlying difficulty for systems of the form \eqref{eq:homogeneous_KK} with a diagonal diffusion matrix $\mathbf{B(U)}=\mathbf{I}$ is that the genuinely nonlinear characteristic field may not yield the necessary compactness estimates required to pass to the vanishing-viscosity limit, especially for the large initial data. This issue becomes particularly non-trivial when the system \eqref{eq:homogeneous_KK} is non-symmetric, as the diffusion acts componentwise and is not necessarily compatible with the characteristic structure of the hyperbolic part. 

For certain two-component nonsymmetric systems with
$V=\phi(u_1u_2)$, a tailored approximation motivated by lubrication theory was recently introduced by Barthwal et al. \cite{barthwal2026global}. Although the approximation proposed in \cite{barthwal2026global} provides the required estimates for the class of systems with $V=\phi(u_1u_2)$, its dissipation mechanism can not be generalized directly to $n\times n$ systems of the form \eqref{eq:homogeneous_KK}, especially due to a rather general nonlinearity in $V$. This motivates us to introduce a different viscous regularization, adapted directly to the nonlinear field $V$, which retains a similar dissipation mechanism while extending the construction to the broader homogeneous class of systems \eqref{eq:homogeneous_KK}.

More precisely, we propose the following viscous regularization for the system \eqref{eq:homogeneous_KK}.
\begin{equation}\label{eq:homogeneous_viscous_system}
    \mathbf U_t^\epsilon
    +\sum_{j=1}^d\partial_{x_j}
       \bigl(\mathbf U^\epsilon V^\epsilon\bigr)
    =
    \epsilon\sum_{j=1}^d\partial_{x_j}
       \bigl(\mathbf U^\epsilon
       \partial_{x_j}V^\epsilon\bigr),
    \qquad
    V^\epsilon:=V(\mathbf U^\epsilon).
\end{equation}
Equivalently, the diffusion matrix associated with
\eqref{eq:homogeneous_viscous_system} is
\[
\mathbf B(\mathbf U)
=
\mathbf U\otimes\nabla_{\mathbf U}V(\mathbf U).
\]
Since the matrix $\mathbf B(\mathbf U)$ has rank one, the system \eqref{eq:homogeneous_viscous_system} is
a partially diffusive viscous system. Nevertheless, using Euler's identity \eqref{eq:Euler_homogeneous}, one can obtain a scalar parabolic equation for the genuinely nonlinear variable $V^\epsilon$.
This is the key structural feature of the proposed regularization. In particular, it provides a triangular viscous structure for the Riemann invariants $(V^\epsilon, \theta_i^\epsilon), \, i=2, 3, \ldots, n$.

Indeed, multiplying $\nabla_{\mathbf U^\eps}V(\mathbf U^\epsilon)$ to
\eqref{eq:homogeneous_viscous_system} and using
\eqref{eq:Euler_homogeneous}, we obtain (see Lemma \ref{lem:decoupled_system} for detailed calculation)
\begin{equation}\label{eq:V_epsilon_equation}
    V_t^\epsilon
    +(k+1)V^\epsilon
       \sum_{j=1}^d\partial_{x_j}V^\epsilon
    =
    \epsilon\left(
        kV^\epsilon\Delta V^\epsilon
        +|\nabla V^\epsilon|^2
    \right).
\end{equation}
Thus, as long as $V^\epsilon$ remains bounded away from zero,
\eqref{eq:V_epsilon_equation} is a uniformly parabolic equation for every fixed
$\epsilon>0$, with principal coefficient
$\epsilon kV^\epsilon>0$. 

For the variables
\begin{equation*}
    \theta_i^\epsilon:=\frac{u_i^\epsilon}{u_1^\epsilon},
    \qquad i=2,\ldots,n,
\end{equation*}
the equations for $u_i^\epsilon$ and $u_1^\epsilon$ give a transport structure of the form (see Lemma \ref{lem:decoupled_system} for detailed calculation).
\begin{equation}\label{eq:xi_epsilon_equation}
    \partial_t\theta_i^\epsilon
    +\sum_{j=1}^d
       \left(
          V^\epsilon-\epsilon\partial_{x_j}V^\epsilon
       \right)
       \partial_{x_j}\theta_i^\epsilon
    =0,
    \qquad i=2,\ldots,n.
\end{equation}
Hence, the viscous system \eqref{eq:homogeneous_viscous_system} decouples into one pure scalar parabolic
equation for $V^\epsilon$ and $n-1$ linear transport equations for
the variables $\theta_i^\epsilon$. Under some technical assumptions (see Theorem \ref{main_theorem_1}), the parabolic equation \eqref{eq:V_epsilon_equation} yields compactness estimates
of $V^\epsilon$, while the associated transport structure \eqref{eq:xi_epsilon_equation}
provides compactness of $\boldsymbol\theta^\epsilon=(\theta_2^\eps, \theta_3^\eps, \ldots, \theta_n^\eps)^\top$. More precisely, these estimates provide a pair $(V,\boldsymbol\theta)\in L^\infty(\Omega_T; \mathbb{R}^n)\cap BV_{loc}(\Omega_T; \mathbb{R}^n)$ such that $V^\epsilon\to V$ in $L^1_{loc}(\Omega_T)$ and $\boldsymbol\theta^\epsilon\to\boldsymbol\theta$ in $(L^1_{loc}(\Omega_T))^{n-1}$. The
reconstruction formula \eqref{eq:reconstruction_formula}
then yields the strong convergence of $\mathbf U^\epsilon$ to a function $\mathbf{U}=(u_1, u_2, \ldots, u_n)^\top$ in $(L^1_{loc}(\Omega_T))^n$. This strong convergence is sufficient to pass to the limit in the nonlinear flux and thereby construct the global weak solution of the Cauchy problem \eqref{eq:homogeneous_KK}-\eqref{initial_Data}, which constitutes the first main result of the article; see Theorem \ref{main_theorem_1}.

To make the discussion precise, we define the weak solution of the Cauchy problem \eqref{eq:homogeneous_KK}-\eqref{initial_Data} as follows.
\begin{definition}[Weak solution of the Cauchy problem \eqref{eq:homogeneous_KK}-\eqref{initial_Data}]\label{weak_soln}
A function $\mathbf{U} \in L^\infty\big(\Omega_T;\,{\mathbb{R}^n}\big)\cap BV_{loc}\big(\Omega_T;\,{\mathbb{R}^n}\big)$ is called a weak solution of the Cauchy problem \eqref{eq:homogeneous_KK}-\eqref{initial_Data} with initial data 
$\mathbf{U}_0 \in L^\infty(\mathbb{R}^d;\,{\mathbb{R}^n})\cap BV_{loc}\big(\mathbb{R}^d;\,{\mathbb{R}^n}\big)$, if for each vector-valued function $  {\bm \varphi} \in C_0^{1} \big( \Omega_T;\,\mathbb{R}^n\big)$, we have the identity 
\begin{equation*}
    \displaystyle\iint_{\Omega_T} 
    \Big( \mathbf{U}\cdot {\bm \varphi}_t
          + (\mathbf{U} V(\mathbf{U}))\cdot {\nabla \bm \varphi}
         \Big)\,dx\,dt
    + \displaystyle\int_{\mathbb{R}^d} \mathbf{U}_0\cdot {\bm \varphi}(\cdot, 0)\,dx = 0.
\end{equation*}
\end{definition}

The viscous approximation \eqref{eq:homogeneous_viscous_system} is
introduced above as a tool for constructing weak solutions of \eqref{eq:homogeneous_KK}-\eqref{initial_Data}. However, a meaningful
regularization should be able to provide more than just the compactness required
for passing to the limit. In particular, it should also reproduce the nonlinear wave
structure of the underlying hyperbolic system. Precisely, in the
zero-viscosity limit, an
admissible shock discontinuity should be obtained from a smooth travelling-wave profile connecting the corresponding end states. The existence of such a viscous profile shows
that the chosen diffusion is compatible with the shock structure of
the inviscid system, while its stability shows that this viscous
selection mechanism is robust under perturbations. This motivates us to analyze the existence and stability of the viscous shock profiles of the viscous system \eqref{eq:homogeneous_viscous_system} as a second important objective of this article.

The existence and stability of viscous shock profiles have played a central role in the theory of viscous conservation laws and associated hyperbolic systems. Il'in and Oleinik \cite{ilin1960asymptotic} first established the stability of viscous shock waves for scalar viscous conservation laws by exploiting the maximum principle. However, this approach does not extend to systems of conservation laws. In the multi-dimensional setting, Goodman \cite{MR978372} demonstrated the stability of weak shocks using the classical anti-derivative method (see also \cite{sahu2026nonlinear}). Subsequently, Hoff and Zumbrun \cite{MR1793680, MR1919784} extended Goodman's results to large shocks via the Green's function approach. 

More recently, the relative entropy framework to analyze shock stability has been utilized by Vasseur and his collaborators \cite{MR2807139,  serre2016relative, vasseur2008recent}. It has emerged as a powerful tool for analyzing shock stability and contraction properties. Unlike Green's function and Evans function methods, the relative entropy method directly exploits the underlying entropy structure of the governing equations, providing a robust methodology for deriving nonlinear stability estimates. Using this method, Kang \cite{MR3953019} established the $L^2$-contraction property of planar viscous shock waves under large perturbations for the multi-dimensional scalar viscous conservation laws with a linear viscosity coefficient. Later, Kang \& Oh \cite{MR4876608} proved the $L^2$-decay of moderate-strength planar viscous shock waves when the flux function is a small perturbation of the Burgers flux. This result was further extended in \cite{MR4195742} up to a shift and a weight to arbitrary strictly convex flux functions in the multi-dimensional scalar setting, though still restricted to linear viscosity. 

In contrast to the available results, a major mathematical challenge in this article is the careful handling of the nonlinear remainder terms caused by the nonlinear viscosity present in the viscous approximation \eqref{eq:homogeneous_viscous_system}. In particular, the relative entropy estimates require delicate handling. To overcome this challenge, we leverage the uniform bounds of the viscous approximation to successfully adapt the relative entropy approach, allowing us to investigate the $L^2$-decay of perturbations around viscous shock waves for the system \eqref{eq:homogeneous_viscous_system}; see Theorem \ref{main_theorem_2}. The proof of Theorem \ref{main_theorem_2} heavily relies on the mathematical structure of the viscous approximation \eqref{eq:homogeneous_viscous_system}. More precisely, the viscous shocks analyzed in this article are restricted to a subset $\widetilde{\mathcal{U}}$ of the state space $\mathcal{U}$, where end states for $\theta_i^\eps$ are same. This restriction allows us to keep the variables $\theta_i^\eps$ constant along the shock
profile, and the travelling-wave dynamics are governed entirely by the scalar parabolic equation for $V^\epsilon$ \eqref{eq:V_epsilon_equation}. In addition to this, we impose a mild assumption (see Assumption \ref{assumption2}) on the homogeneous velocity $V(\mathbf{U})$ to ensure the $L^2$-contraction of the shock profile. This indicates that the same parabolic–transport decomposition that provides compactness in the vanishing-viscosity limit also governs the structure and stability of the viscous shock. This constitutes the second main result of this article (see Theorem \ref{main_theorem_2}), which demonstrates that the adapted regularization \eqref{eq:homogeneous_viscous_system} is not only an analytical approximation of \eqref{eq:homogeneous_KK}, but also a rigorous mechanism for selecting dynamically stable shock waves. As a direct corollary of this, we establish the stability of the viscous shock for scalar conservation laws with nonlinear viscosity, thereby generalizing the classical results, which are limited to linear viscosity (see Lemma \ref{lem:forW} and Remark \ref{rem:forviscousshock}).

The remainder of the article is structured as follows. In Section \ref{sec: hyperbolicity}, we discuss the basic properties of the system \eqref{eq:homogeneous_KK}.
Next, in Section \ref{sec: viscous}, we propose our novel approximation \eqref{eq:homogeneous_viscous_system}  and deduce several a priori bounds to prove the existence of global smooth solutions of the Cauchy problem for the approximation system \eqref{eq:homogeneous_viscous_system}. In Section \ref{sec: vanishing}, we utilize these a priori bounds to ensure the well-posedness of the global weak solution
for \eqref{eq:homogeneous_KK}-\eqref{initial_Data} and prove Theorem \ref{main_theorem_1} as the first main result of this article. Section \ref{sec: viscous_shock} is devoted to analyzing the existence and stability of viscous shocks for the system \eqref{eq:homogeneous_viscous_system}, which is summarized in Theorem \ref{main_theorem_2} as the second main result of this article. Concluding remarks and future outlook are provided in Section \ref{sec: conclusions}. 
\subsection*{Notations}\label{Notation}
Throughout the article, we use the following standard notation. We denote by $C^\ell(\mathcal{D})$ the space of functions on a domain $\mathcal{D}$ whose derivatives up to order $\ell$ are continuous, where $\ell \in \mathbb{N}_0$. \\
For $\alpha \in (0,1)$, $C_b^\alpha(\mathcal{D};\mathcal{U})$ denotes the space of bounded, $\alpha$-H\"older continuous functions $f \colon \mathcal{D} \to \mathcal{U} \subset \mathbb{R}^m$. This space is endowed with the norm
\[
    {\|f\|}_{C_b^\alpha (\mathcal{D})} = {\|f\|}_{L^\infty(\mathcal{D})} + {[f]}_{C^\alpha(\mathcal{D})},
\]
where the H\"older seminorm is defined as
\[
    {[f]}_{C^\alpha(\mathcal{D})} = \sup_{\substack{x,y \in \mathcal{D} \\ x \neq y}} \frac{|f(x)-f(y)|}{|x-y|^\alpha}.
\]

More generally, for any $\ell \in \mathbb{N}_0$, $C_b^{\ell+\alpha}(\mathcal{D};\mathcal{U})$ denotes the space of $C^\ell$ functions $f \colon \mathcal{D} \to \mathcal{U}$ such that $f$ and all of its partial derivatives up to order $k$ are bounded, and its $\ell$-th order partial derivatives are $\alpha$-H\"older continuous. 

This space is endowed with the norm
\[
    {\|f\|}_{C_b^{\ell+\alpha}(\mathcal{D})} = \sum_{|j| \le \ell} {\|\partial^j f\|}_{L^\infty(\mathcal{D})} + \sum_{|j| = \ell} {[\partial^j f]}_{C^\alpha(\mathcal{D})},
\]
where $j$ is a multi-index.

Further, we denote by $BV_{\mathrm{loc}}(\mathcal{D})$ the space of functions of locally bounded variation. A function $f \in L^1_{\mathrm{loc}}(\mathcal{D})$ belongs to $BV_{\mathrm{loc}}(\mathcal{D})$ if, for every compact subset $K \subset \mathcal{D}$, the total variation of $f$ over $K$ is finite, i.e.,
\[
    {BV}_{\mathrm{loc}}(\mathcal{D}) := \left\{ f \in L^1_{\mathrm{loc}}(\mathcal{D}) : \sup_{\substack{\psi \in C_c^1(K;\mathbb{R}^n) \\ \|\psi\|_{L^\infty(K)} \le 1}} \int_{\mathcal{D}} f(x) \operatorname{div} \psi(x) \, \mathrm{d}x < \infty \quad \text{for all compact } K \subset \mathcal{D} \right\}.
\]
Further, for $p \in [1,\infty]$, we denote the Lebesgue space on $\mathcal{D}$ as $L^p(\mathcal{D})$ and the $L^p(\mathcal{D})$-norm as ${\lVert \cdot \rVert}_{L^p(\mathcal{D})}$, which for a function $f \in L^p(\mathcal{D})$ is defined as
\begin{align*}
    {\lVert f\rVert}_{L^p(\mathcal{D})} = \biggl(\int_{\mathcal{D}} |f(x)|^p \,dx\biggr)^{\frac{1}{p}} ~ \text{for}~p \in[1,\infty),
\end{align*}
and 
\[
{\lVert f\rVert}_{L^\infty(\mathcal{D})}
:=
\operatorname*{ess\,sup}_{x\in \mathcal{D}}|f(x)|.
\]
Finally, for a Banach space $X$ and $T>0$, we denote by $L^p(0, T;X)$, the Bochner space of measurable functions $f:(0,T)\to X$ with an induced norm
\[
    {\lVert f \rVert}_{L^p(0,T;X)}
    = \biggl( \int_0^T {\lVert f(\cdot, t) \rVert}_{X}^p \,\mathrm{d}t \biggr)^{1/p},
\]
for $p \in [1,\infty)$, while for $p = \infty$
\[
{\lVert f \rVert}_{L^{\infty}(0,T;X)}
    = \operatorname*{ess\,sup}_{t\in[0, T]} {\lVert f(\cdot,t) \rVert}_{X}.
\]
For vector-valued functions, these spaces are interpreted componentwise. 

Finally, throughout the article $C > 0$ denotes a generic positive constant independent of time $t>0$ and the smallness parameters $\epsilon$ and $\delta$ that may change from line to line, unless otherwise stated.
\section{Riemann invariant structure and hyperbolicity of the system \eqref{eq:homogeneous_KK}}\label{sec: hyperbolicity}
In this section, we first examine the characteristic structure of the inviscid system
\eqref{eq:homogeneous_KK}. For each $j=1,\ldots,d$, define
\[
\mathbf F_j(\mathbf U):=\mathbf U V(\mathbf U).
\]
Since all spatial fluxes have the same form, their Jacobian matrices
coincide and take the following form
\begin{equation*}\label{eq:flux_jacobian}
    \mathbf A(\mathbf U)
    :=
    D_{\mathbf U}\mathbf F_j(\mathbf U)
    =
    V(\mathbf U)\mathbf I
    +
    \mathbf U\otimes\nabla_{\mathbf U}V(\mathbf U).
\end{equation*}
Therefore, for smooth solutions, the system
\eqref{eq:homogeneous_KK} can be rewritten as
\begin{equation}\label{eq:homogeneous_KK_quasilinear}
    \mathbf U_t
    +
    \mathbf A(\mathbf U)
    \sum_{j=1}^d\partial_{x_j}\mathbf U
    =0.
\end{equation}
For a unit vector
$\boldsymbol\nu=(\nu_1,\ldots,\nu_d)^\top\in\mathbb R^d$, the
directional Jacobian matrix is then defined as
\begin{equation}\label{eq:directional_symbol}
    \mathbf A_{\boldsymbol\nu}(\mathbf U)
    :=
    \sum_{j=1}^d
    \nu_jD_{\mathbf U}\mathbf F_j(\mathbf U)
    =
    \sigma_{\boldsymbol\nu}\mathbf A(\mathbf U),
    \qquad
    \sigma_{\boldsymbol\nu}
    :=
    \sum_{j=1}^d\nu_j.
\end{equation}
In what follows, we prove that the matrix $\mathbf{A_{\nu}(U)}$ is diagonalizable over $\mathbb{R}$ and thus the system \eqref{eq:homogeneous_KK} is hyperbolic \cite{dafermos2005hyperbolic}.
\begin{lemma}[Eigenstructure and hyperbolicity]
\label{lem:eigenstructure_homogeneous_system}
Let Assumption~\ref{assumption_state_space} holds. Then, for
every $\mathbf U\in\mathcal U$, the matrix
$\mathbf A(\mathbf U)$ is diagonalizable over $\mathbb R$ with real eigenvalues given by
\begin{align}
    \lambda_{\mathrm{ld}}(\mathbf U)
    &=
    V(\mathbf U),
    &&\text{with multiplicity }n-1,
    \label{eq:linearly_degenerate_eigenvalue}
    \\
    \lambda_{\mathrm{g}}(\mathbf U)
    &=
    (k+1)V(\mathbf U),
    &&\text{with multiplicity }1.
    \label{eq:genuinely_nonlinear_eigenvalue}
\end{align}
The corresponding right eigenspaces are
\begin{align}
    E_{\mathrm{ld}}(\mathbf U)
    &=
    \ker \nabla_{\mathbf U}V(\mathbf U),
    \label{eq:tangential_eigenspace}
    \\
    E_{\mathrm{g}}(\mathbf U)
    &=
    \operatorname{span}\{\mathbf U\}.
    \label{eq:radial_eigenspace}
\end{align}

More precisely, if $\mathbf e_i$ denotes the $i$-th standard
basis vector of $\mathbb R^n$, then a basis of right eigenvectors of $\mathbf{A(U)}$ is
given by
\begin{equation}\label{eq:explicit_right_eigenvectors}
    \mathbf r_{\mathrm{g}}(\mathbf U)=\mathbf U,
    \qquad
    \mathbf r_i(\mathbf U)
    =
    \mathbf e_i
    -
    \frac{\partial_{u_i}V(\mathbf U)}
         {kV(\mathbf U)}
    \mathbf U,
    \quad i=2,3, \ldots,n.
\end{equation}
Here, $\mathbf r_{\mathrm{g}}$ corresponds to
$\lambda_{\mathrm{g}}$, while
$\mathbf r_2, \mathbf r_3,\ldots,\mathbf r_n$ correspond to
$\lambda_{\mathrm{ld}}$.

For every unit vector $\boldsymbol\nu\in\mathbb R^d$, the eigenvalues
of the directional Jacobian matrix $\mathbf A_{\boldsymbol\nu}$ are given by
\begin{align}
    \lambda_{\mathrm{ld}}^{\boldsymbol\nu}(\mathbf U)
    &=
    \sigma_{\boldsymbol\nu}V(\mathbf U),
    \label{eq:directional_ld_eigenvalue}
    \\
    \lambda_{\mathrm{g}}^{\boldsymbol\nu}(\mathbf U)
    &=
    \sigma_{\boldsymbol\nu}(k+1)V(\mathbf U).
    \label{eq:directional_g_eigenvalue}
\end{align}
In particular, the system \eqref{eq:homogeneous_KK} is hyperbolic.
\end{lemma}
\begin{proof}
For any $\mathbf z\in\mathbb R^n$, the action of the flux Jacobian $\mathbf A(\mathbf U)$ on $\mathbf z$ is given by
\begin{equation}\label{eq:A_action}
    \mathbf A(\mathbf U)\mathbf z = V(\mathbf U)\mathbf z + \mathbf U \bigl( \nabla_{\mathbf U}V(\mathbf U)\cdot\mathbf z \bigr).
\end{equation}
Since $V$ is homogeneous of degree $k$, Euler's identity \eqref{eq:Euler_homogeneous} and Assumption~\ref{assumption_state_space} implies that $kV(\mathbf U)>0$ for each $\mathbf{U}\in \mathcal{U}$. This shows that $\nabla_{\mathbf U}V(\mathbf U)\cdot\mathbf U > 0$ and thus $\nabla_{\mathbf U}V(\mathbf U)\neq\mathbf 0$.

Taking $\mathbf z=\mathbf U$ in \eqref{eq:A_action}, we obtain
\[
\mathbf A(\mathbf U)\mathbf U = V(\mathbf U)\mathbf U + \mathbf U \bigl(kV(\mathbf U)\bigr) = (k+1)V(\mathbf U)\mathbf U.
\]
Thus, $\mathbf U$ is a right eigenvector corresponding to the eigenvalue $\lambda_{\mathrm{g}}=(k+1)V(\mathbf U)$, which spans the 1-dimensional eigenspace $E_{\mathrm{g}}(\mathbf U) = \operatorname{span}\{\mathbf U\}$. This proves \eqref{eq:genuinely_nonlinear_eigenvalue} and \eqref{eq:radial_eigenspace}.

Next, let $L:\mathbb R^n\to\mathbb R$ be the linear functional defined by
$L(\mathbf z):=\nabla_{\mathbf U}V(\mathbf U)\cdot\mathbf z$. Since
$\nabla_{\mathbf U}V(\mathbf U)\neq\mathbf 0$ for $\mathbf U\in\mathcal U$,
$L$ is surjective and thus $\operatorname{rank}(L)=1$. Using the
rank--nullity Theorem, we have
\[
\dim\ker\nabla_{\mathbf U}V(\mathbf U) = n-1.
\]
In what follows, we prove that $E_{\mathrm{ld}}(\mathbf U)=\ker\nabla_{\mathbf U}V(\mathbf U)$. For any $\mathbf z\in\ker\nabla_{\mathbf U}V(\mathbf U)$,
\eqref{eq:A_action} simplifies to
\[
\mathbf A(\mathbf U)\mathbf z = V(\mathbf U)\mathbf z.
\] 
This proves \eqref{eq:linearly_degenerate_eigenvalue} and \eqref{eq:tangential_eigenspace}.

Furthermore, since $\nabla_{\mathbf U}V(\mathbf U)\cdot\mathbf U = kV(\mathbf U) > 0$, it immediately follows that
\[
\mathbf U\notin \ker\nabla_{\mathbf U}V(\mathbf U).
\]
This implies
\[
E_{\mathrm{g}}(\mathbf U)\cap E_{\mathrm{ld}}(\mathbf U)=\{\mathbf{0}\}.
\] 
Therefore, we can express the space $\mathbb{R}^n$ as
\[
\mathbb R^n = \operatorname{span}\{\mathbf U\} \oplus \ker\nabla_{\mathbf U}V(\mathbf U),
\]
which proves that $\mathbf A(\mathbf U)$ has a complete basis of eigenvectors and is therefore diagonalizable over $\mathbb R$.

To construct an explicit basis for $\ker\nabla_{\mathbf U}V(\mathbf U)$, we define the vectors $\mathbf r_i(\mathbf U)$ for $i=2,\ldots,n$ as in \eqref{eq:explicit_right_eigenvectors}. Applying Euler's identity \eqref{eq:Euler_homogeneous} again gives
\begin{align*}
\nabla_{\mathbf U}V(\mathbf U)\cdot\mathbf r_i(\mathbf U) &= \partial_{u_i}V(\mathbf U) - \frac{\partial_{u_i}V(\mathbf U)}{kV(\mathbf U)} \bigl( \nabla_{\mathbf U}V(\mathbf U)\cdot\mathbf U \bigr) \\
&= \partial_{u_i}V(\mathbf U) - \frac{\partial_{u_i}V(\mathbf U)}{kV(\mathbf U)} \bigl( kV(\mathbf U) \bigr)=0.
\end{align*}
Hence, $\mathbf r_i\in E_{\mathrm{ld}}(\mathbf U)$ for each $i=2,\ldots,n$ in view of \eqref{eq:tangential_eigenspace}. Further, in view of Assumption \ref{assumption_state_space}, one can easily verify that the set $\{\mathbf{r}_i\}, \, {i=2, 3, \ldots, n}$ is linearly independent for $\mathbf{U}\in \mathcal{U}$ and thus forms a basis for $\ker\nabla_{\mathbf U}V(\mathbf U)$.

Now, for each $\boldsymbol\nu=(\nu_1,\ldots,\nu_d)^\top\in\mathbb R^d$, the eigenvalues of the directional Jacobian  $\mathbf{A}_\nu$ are $\sigma_{\mathbf{\nu}}$ multiple of eigenvalues of $\mathbf{A}(\mathbf{U})$ in view of \eqref{eq:directional_symbol}. This proves \eqref{eq:directional_ld_eigenvalue}-\eqref{eq:directional_g_eigenvalue}. Moreover, using \eqref{eq:directional_symbol}, the right eigenspaces are the same as those of the matrix $\mathbf{A}(\mathbf{U})$. Hence, the system \eqref{eq:homogeneous_KK} is hyperbolic. This completes the proof of the Lemma.
\end{proof}
For a given unit vector $\boldsymbol\nu=(\nu_1,\ldots,\nu_d)^\top\in\mathbb R^d$ and for every $\mathbf{r}_i\in E_{\text{ld}}(\mathbf{U})$, we have
\begin{align*}
    \nabla_{\mathbf{U}}\lambda^{\boldsymbol\nu}_{\text{ld}}(\mathbf{U}).\mathbf{r}_i = \sigma_{\boldsymbol\nu}\nabla_{\mathbf{U}}V(\mathbf{U}).\mathbf{r}_i = 0.
\end{align*}
Thus, characteristic fields corresponding to the eigenvalue $\lambda^{{\boldsymbol\nu}}_{\text{ld}}(\mathbf{U})$ are linearly degenerate. Meanwhile, we compute
\begin{align*}
    \nabla_{\mathbf{U}}\lambda^{{\boldsymbol\nu}}_{\text{g}}(\mathbf{U}).\mathbf{r}_{\text{g}} &= \sigma_{\boldsymbol\nu}(k+1)\nabla_{\mathbf{U}}V(\mathbf{U}).\mathbf{U}\\
    & = \sigma_{\boldsymbol\nu}k(k+1)V(\mathbf{U}) \neq 0 \quad \text{in}~~\mathcal{U}.
\end{align*}
Therefore, the characteristic field corresponding to the eigenvalue $\lambda^{{\boldsymbol\nu}}_{\text{g}}(\mathbf{U})$ is genuinely nonlinear. We now identify the associated Riemann invariants. Since the eigenspaces of $\mathbf{A}_{\boldsymbol\nu}$ are independent of $\boldsymbol\nu$, the Riemann invariant structure remains identical in every direction. Thus, we restrict to the one-dimensional case. 

By definition, a smooth function $\mathbf{\Pi}(\mathbf{U})$ is a Riemann invariant corresponding to a right eigenvector $\mathbf{r}(\mathbf{U})$ if 
\[
\nabla_{\mathbf{U}} \mathbf{\Pi}(\mathbf{U})\cdot \mathbf{r}(\mathbf{U})=0.
\]
Therefore, one can easily obtain the complete set of Riemann invariants for the system \eqref{eq:homogeneous_KK} as 
\begin{align}
    \Pi_1 = V(\mathbf{U}), \quad \Pi_i = \theta_i = \frac{u_i}{u_1}, \quad \text{for}~~i = 2,\ldots,n.
\end{align}
 
\section{An admissible approximation of the  system \eqref{eq:homogeneous_KK}}\label{sec: viscous}
We introduce an admissible approximation adapted to the structure of \eqref{eq:homogeneous_KK} in this Section. The approximation is designed to provide the dissipation estimates required for the vanishing-viscosity limit (see Section \ref{sec: vanishing}) and an invariant-region structure for the inviscid system \eqref{eq:homogeneous_KK}. The proposed viscous approximation for the system \eqref{eq:homogeneous_KK} is given by
\begin{equation}\label{eq:homogeneous_KK_viscos_2}
  \mathbf{U}_t^\epsilon + \sum_{j=1}^d\partial_{x_j} \bigl(\mathbf{U}^\epsilon V^\epsilon\bigr) = \epsilon\sum_{j=1}^d\partial_{x_j} \bigl(\mathbf{U}^\epsilon \partial_{x_j}V^\epsilon\bigr), \qquad V^\epsilon:=V(\mathbf{U}^\epsilon),
\end{equation}
subject to the initial data 
\begin{align}\label{eq:homogeneous_KK_viscos_initial_data}
\mathbf{U}^\epsilon(\cdot, 0) = \mathbf{U}^\epsilon_0 \quad \text{in } \mathbb{R}^d.
\end{align}
In \eqref{eq:homogeneous_KK_viscos_initial_data}, $\mathbf{U}^\epsilon_0$ is obtained by regularizing the initial data $\mathbf{U}_0$ defined in \eqref{initial_Data} by convolving with the Friedrichs mollifier. Precisely, 
\[
\mathbf{U}_0^\epsilon = j^\epsilon * \mathbf{U}_0.
\]
Here $j^\epsilon$ is a mollifier satisfying
\[
j^\epsilon(x) = \frac{1}{\epsilon^d}j\left(\frac{x}{\epsilon}\right),
\]
with
\[
j(x) = \begin{cases}
    \dfrac{1}{A}\exp\left\{\dfrac{1}{|x|^2-1}\right\}, & \text{if } |x| < 1, \\[2ex]
    0, & \text{if } |x| \geq 1,
\end{cases}
\]
and $A$ is chosen such that $\int_{\mathbb{R}^d} j(x) \, dx = 1$. One can verify that $\{\mathbf{U}^\epsilon_0\}_{\epsilon > 0} \subset C^\infty(\mathbb{R}^d; \mathcal{U}) \cap L^\infty(\mathbb{R}^d; \mathcal{U})$ satisfies for $\epsilon \to 0$ (see \cite{Evans})
\[
\mathbf{U}_0^\epsilon \to \mathbf{U}_0 \quad \text{a.e. and in } L^p_{\text{loc}}(\mathbb{R}^d) \text{ for } p \in [1, \infty).
\]
In the following Lemma, we first prove that the system \eqref{eq:homogeneous_KK_viscos_2} can be converted into a triangular form using the Riemann invariants $(V^\eps, \bm{\theta}^\eps)$.
\begin{lemma}\label{lem:decoupled_system}
Let $\mathbf U^\epsilon$ be a smooth solution of the Cauchy problem
\eqref{eq:homogeneous_KK_viscos_2}-\eqref{eq:homogeneous_KK_viscos_initial_data} with $\mathbf{U}^\eps\in \mathcal{U}$. Then the functions
$V^\epsilon$ and $\theta_i^\epsilon=\frac{u_i^\epsilon}{u_1^\epsilon},
\, \,i=2,\ldots,n,$ satisfy
\begin{align}
V^\epsilon_t
+
(k+1)V^\epsilon
\sum_{j=1}^d\partial_{x_j}V^\epsilon
&=
\epsilon\left(
kV^\epsilon\Delta V^\epsilon
+
|\nabla V^\epsilon|^2
\right),
\label{eq:scalar_velocity_equation_main}\\
\theta^\epsilon_{it}
+
\sum_{j=1}^d
\left(
V^\epsilon-\epsilon\partial_{x_j}V^\epsilon
\right)
\partial_{x_j}\theta_i^\epsilon
&=0,
\qquad i=2,\ldots,n.
\label{eq:ratio_transport_equation_main}
\end{align}
\end{lemma}

\begin{proof}
Expanding the two divergence terms in
\eqref{eq:homogeneous_KK_viscos_2}, we can express \eqref{eq:homogeneous_KK_viscos_2} as
\begin{equation}\label{eq:expanded_approximate_system}
\mathbf U^\epsilon_t
+
\sum_{j=1}^d
\left(
V^\epsilon-\epsilon\partial_{x_j}V^\epsilon
\right)
\partial_{x_j}\mathbf U^\epsilon
=
\mathbf U^\epsilon
\left(
\epsilon\Delta V^\epsilon
-
\sum_{j=1}^d\partial_{x_j}V^\epsilon
\right).
\end{equation}

Taking the scalar product of
\eqref{eq:expanded_approximate_system} with
$\nabla_{\mathbf U^\eps}V(\mathbf U^\epsilon)$ and using
\[
V^\epsilon_t
=
\nabla_{\mathbf U^\eps}V(\mathbf U^\epsilon)
\cdot\mathbf U^\epsilon_t,
\qquad
\partial_{x_j}V^\epsilon
=
\nabla_{\mathbf U^\eps}V(\mathbf U^\epsilon)
\cdot\partial_{x_j}\mathbf U^\epsilon,
\]
we obtain
\begin{align}\label{V_first}
V^\epsilon_t
+
\sum_{j=1}^d
\left(
V^\epsilon-\epsilon\partial_{x_j}V^\epsilon
\right)
\partial_{x_j}V^\epsilon
&=
\left(
\nabla_{\mathbf U^\eps}V(\mathbf U^\epsilon)
\cdot\mathbf U^\epsilon
\right)
\left(
\epsilon\Delta V^\epsilon
-
\sum_{j=1}^d\partial_{x_j}V^\epsilon
\right).
\end{align}
Using \eqref{eq:Euler_homogeneous}, \eqref{V_first} reduces to
\begin{align}\label{V_second}
V^\epsilon_t
+
V^\epsilon\sum_{j=1}^d\partial_{x_j}V^\epsilon
-
\epsilon|\nabla V^\epsilon|^2
&=
\epsilon kV^\epsilon\Delta V^\epsilon
-
kV^\epsilon
\sum_{j=1}^d\partial_{x_j}V^\epsilon.
\end{align}
Hence, rearranging \eqref{V_second} yields \eqref{eq:scalar_velocity_equation_main}.

For $i=2,\ldots,n$, the componentwise form of
\eqref{eq:expanded_approximate_system} gives
\[
u^\epsilon_{it}
+
\sum_{j=1}^d
\left(
V^\epsilon-\epsilon\partial_{x_j}V^\epsilon
\right)
\partial_{x_j}u_i^\epsilon
=
u_i^\epsilon
\left(
\epsilon\Delta V^\epsilon
-
\sum_{j=1}^d\partial_{x_j}V^\epsilon
\right).
\]
Thus, we directly compute
\begin{align*}
&\theta^\epsilon_{it}
+
\sum_{j=1}^d
\left(
V^\epsilon-\epsilon\partial_{x_j}V^\epsilon
\right)
\partial_{x_j}\theta_i^\epsilon\\
&\quad=
\frac{1}{(u_1^\epsilon)^2}
\Bigg[
u_1^\epsilon
\left(
u^\epsilon_{it}
+
\sum_{j=1}^d
\left(
V^\epsilon-\epsilon\partial_{x_j}V^\epsilon
\right)
\partial_{x_j}u_i^\epsilon
\right)-
u_i^\epsilon
\left(
u^\epsilon_{1t}
+
\sum_{j=1}^d
\left(
V^\epsilon-\epsilon\partial_{x_j}V^\epsilon
\right)
\partial_{x_j}u_1^\epsilon
\right)
\Bigg].
\end{align*}
Both terms on the right-hand side contain the common factor $\epsilon\Delta V^\epsilon
-
\sum_{j=1}^d\partial_{x_j}V^\epsilon$ and therefore cancel. This proves
\eqref{eq:ratio_transport_equation_main}.
\end{proof}
With this mathematical structure in place, we now prove that the Cauchy problem \eqref{eq:homogeneous_KK_viscos_2}-\eqref{eq:homogeneous_KK_viscos_initial_data} has a unique global-in-time solution $\mathbf{U}^\eps$ and an $\epsilon$-independent invariant domain.
\begin{proposition}[Global well-posedness of the Cauchy problem \eqref{eq:homogeneous_KK_viscos_2}-\eqref{eq:homogeneous_KK_viscos_initial_data} and an invariant domain]\label{lem:global_viscous}
Let $\epsilon>0$, $T>0$, and $\alpha\in(0,1)$ and suppose that $\mathbf U^\epsilon_0
\in
C_b^{2+\alpha}(\mathbb R^d;\mathcal{U})$ with
\begin{align}\label{eq: U0bound}
    \mathbf U_0^\epsilon \in [m, M]^n \subset \mathcal{U}, \quad 0<m<M .
\end{align}
Then the Cauchy problem \eqref{eq:homogeneous_KK_viscos_2}-\eqref{eq:homogeneous_KK_viscos_initial_data} admits a unique classical solution on $\mathbb R^d\times[0,T]$. More precisely,
\[
\mathbf U^\epsilon
\in
C\bigl([0,T];C_b^{1+\alpha}(\mathbb R^d;\mathcal{U})\bigr)
\cap
C^1\bigl([0,T];C_b^\alpha(\mathbb R^d;\mathcal{U})\bigr),
\]
and
\[
V^\epsilon:=V(\mathbf U^\epsilon)
\in
C\bigl([0,T];C_b^{2+\alpha}(\mathbb R^d)\bigr)
\cap
C^1\bigl([0,T];C_b^\alpha(\mathbb R^d)\bigr).
\]
Define
\begin{equation}\label{eq:constantm_0M_0}
    m_0:=\min_{\mathbf U_0^\eps\in{[m,M]}^n}V(\mathbf U_0^\eps),\qquad M_0:=\max_{\mathbf U_0^\eps\in{[m,M]}^n}V(\mathbf U_0^\eps).
\end{equation}
Then
\[
0<m_0
\leq V(\mathbf U^\epsilon(x,t))
\leq M_0<\infty\, \quad \forall\, (x,t)\in\mathbb R^d\times[0,T].
\]
Moreover, there exist $m_*, M_*>0$ independent of $\epsilon$ such that
\[
\mathbf U^\epsilon\in {[m_*, M_*]}^n\, \quad \forall\, (x,t)\in\mathbb R^d\times[0,T].
\]
In particular, the set 
\begin{align}\label{invariant_domain}
    \mathcal{I}=\bigg\{\mathbf{U}^\epsilon\in \mathcal{U}:\, m_0\leq V(\mathbf{U}^\epsilon)\leq M_0, \quad \dfrac{m}{M}\leq \theta_i\leq\dfrac{M}{m}\bigg\}
\end{align}
is an invariant domain for \eqref{eq:homogeneous_KK_viscos_2}.
\end{proposition}

\begin{proof}
From Lemma \ref{lem:decoupled_system}, it is evident that the considered viscous approximated system \eqref{eq:homogeneous_KK_viscos_2} can be recasted in the variable $V^\epsilon$ and $\theta_{i}^\epsilon:=\frac{u_i^\epsilon}{u_1}, \, i=2, 3, \ldots, n$ as \eqref{eq:scalar_velocity_equation_main}-\eqref{eq:ratio_transport_equation_main}.
Further, we define the corresponding initial data
\begin{align}
V^\epsilon_0(x)&=V(\mathbf U^\epsilon_0(x)), \label{eq:V_initial_data}\\
\theta_{i, 0}^\epsilon(x)&=\dfrac{u_{i, 0}^\epsilon(x)}{u_{1, 0}^\epsilon(x)}\label{eq:xi_initial_data}.
\end{align}
The composition theorem for Hölder functions gives
\[
V^\epsilon_0
\in C_b^{2+\alpha}(\mathbb R^d).
\]
Now define
\begin{align}\label{eq: VtoQ}
Q^\epsilon
:=
(V^\epsilon)^{\frac{k+1}{k}},
\qquad
Q^\epsilon_0
:=
(V^\epsilon_0)^{\frac{k+1}{k}}.
\end{align}
Then the initial data \eqref{eq:V_initial_data} becomes
\begin{equation}\label{eq:Q_epsilon_initial_data}
Q^\epsilon(\cdot,0)=Q^\epsilon_0=(V_0^\epsilon)^{\frac{k+1}{k}}.
\end{equation}
Because $V^\epsilon_0$ is bounded away from zero and belongs to
$C_b^{2+\alpha}(\mathbb R^d)$, we have
\[
Q^\epsilon_0\in C_b^{2+\alpha}(\mathbb R^d).
\]
Moreover, the bounds of $V_0^\eps$ implies
\[
m_0^{\frac{k+1}{k}}
\leq
Q^\epsilon_0
\leq
M_0^{\frac{k+1}{k}}.
\]
We now derive the evolution equation of $Q^\epsilon$. Set
\[
W^\epsilon:=(V^\epsilon)^{1/k}.
\]
Then
\[
Q^\epsilon=(W^\epsilon)^{k+1},
\]
and the equation for $W^\epsilon$ is obtained as
\begin{equation}\label{W_equation}
W^\epsilon_t
+
\sum_{j=1}^d
\partial_{x_j}(W^\epsilon)^{(k+1)}
=
\epsilon\frac{k}{k+1}\Delta (W^\epsilon)^{k+1}.
\end{equation}
It follows that $Q^\epsilon$ satisfies
\begin{equation}\label{eq:Q_epsilon}
Q^\epsilon_t
+
(k+1)(Q^\epsilon)^{\frac{k}{k+1}}
\sum_{j=1}^d\partial_{x_j}Q^\epsilon
=
\epsilon k
(Q^\epsilon)^{\frac{k}{k+1}}
\Delta Q^\epsilon.
\end{equation}
The diffusion coefficient in \eqref{eq:Q_epsilon} is
\[
a_\epsilon(Q)
:=
\epsilon kQ^{\frac{k}{k+1}}.
\]
It is smooth and strictly positive for $Q>0$. Since $Q^\epsilon_0$ is bounded away from zero, standard local
quasilinear parabolic theory (see, e.g., \cite{amann1995linear, ladyparabolic}) gives a unique classical solution
$Q^\epsilon$ of
\eqref{eq:Q_epsilon}--\eqref{eq:Q_epsilon_initial_data} on a maximal
interval $[0,T_{\max})$, $T_{\max}>0$, such that
\[
Q^\epsilon
\in
C\bigl([0,T_{\max});C_b^{2+\alpha}(\mathbb R^d)\bigr)
\cap
C^1\bigl([0,T_{\max});C_b^\alpha(\mathbb R^d)\bigr).
\]

Since
\[
Q^\epsilon_0(x)
\geq
m_0^{\frac{k+1}{k}}>0
\qquad
\text{for every }x\in\mathbb R^d,
\]
and
\[
Q^\epsilon
\in
C\bigl([0,T_{\max});C_b(\mathbb R^d)\bigr),
\]
continuity at $t=0$ implies that there exists $\tau_1>0$ such that
\[
{\left\|
Q^\epsilon(\cdot,t)-Q^\epsilon(\cdot,0)
\right\|}_{L^\infty(\mathbb R^d)}
<
\delta_Q
\qquad
\text{for }0\leq t\leq\tau_1,
\]
where
\[
\delta_Q
=
\min\left\{\frac{(m_0)^\frac{k+1}{k}}{2},1\right\}.
\]
Consequently,
\[
\begin{aligned}
Q^\epsilon(x,t)
&\geq
Q^\epsilon_0(x)-\delta_Q
\\
&\geq
(m_0)^\frac{k+1}{k}-\frac{(m_0)^\frac{k+1}{k}}{2}\\
&=
\frac{(m_0)^\frac{k+1}{k}}{2}>0,
\end{aligned}
\]
Therefore, \eqref{eq:Q_epsilon} is a uniformly parabolic equation in $[0, \tau_1]$, and thus the standard maximum principle can be applied on $Q^\epsilon$, which gives the bounds 
\[
Q^\epsilon\in [m_0^\frac{k+1}{k}, M_0^\frac{k+1}{k}]\quad \forall \,(x, t)\in \mathbb{R}^d\times [0, \tau_1].
\]
The coefficients of the equation and all their derivatives with
respect to $Q$ therefore remain bounded on the range of the
solution. The standard continuation theorem for uniformly
parabolic quasilinear equations then excludes a finite maximal
existence time. Hence
\[
T_{\max}=\infty.
\]
Moreover, using \eqref{eq: VtoQ}, we have
\[
V^\epsilon
\in
C\bigl([0,T];C_b^{2+\alpha}(\mathbb R^d)\bigr)
\cap
C^1\bigl([0,T];C_b^\alpha(\mathbb R^d)\bigr),
\]
and the preceding calculations show that $V^\epsilon$ solves
\eqref{eq:scalar_velocity_equation_main} and \eqref{eq:V_initial_data}. Moreover, the bounds on $Q^\epsilon$ imply that $V^\epsilon\in [m_0, M_0]$. 

Now it remains to show that the Cauchy problem \eqref{eq:ratio_transport_equation_main} and \eqref{eq:xi_initial_data} has a unique solution $\theta_i^\eps$ for $i=2, 3, \ldots, n$. Since, for each $i=1, 2, \ldots, n$, we have by \eqref{eq: U0bound} the bounds
\[
m\leq u^\epsilon_{i,0},u^\epsilon_{1,0}\leq M.
\]
Therefore, we have
\[
\frac{m}{M}
\leq
\theta^\epsilon_{i,0}(x)
\leq
\frac{M}{m}.
\]
Now, define
\[
\mathbf c^\epsilon
:=
\left(
V^\epsilon-\epsilon\partial_{x_1}V^\epsilon,
\ldots,
V^\epsilon-\epsilon\partial_{x_d}V^\epsilon
\right).
\]
The regularity of $V^\epsilon$ thus implies
\[
\mathbf c^\epsilon
\in
C\bigl([0,T];
C_b^{1+\alpha}(\mathbb R^d)\bigr).
\]
Hence $\mathbf c^\epsilon$ is bounded and globally Lipschitz in
$x$, uniformly for $t\in[0,T]$ and thus the characteristic flow is globally defined on $[0,T]$. The method of characteristics thus
implies the existence of a unique classical solution  $\theta_i^\epsilon$
of the Cauchy problem \eqref{eq:ratio_transport_equation_main}-\eqref{eq:xi_initial_data}, which satisfy
\[
\theta_i^\epsilon
\in
C\bigl([0,T];C_b^{1+\alpha}(\mathbb R^d)\bigr)
\cap
C^1\bigl([0,T];C_b^\alpha(\mathbb R^d)\bigr)
\]
and
\begin{equation}\label{eq:global_ratio_bounds}
\frac{m}{M}
\leq
\theta_i^\epsilon(x,t)
\leq
\frac{M}{m}.
\end{equation}
It remains to reconstruct $\mathbf U^\epsilon$. Set
\[
I_{M}
=
\left[
\frac{m}{M},\frac{M}{m}
\right]^{n-1}.
\]
Since $g>0$ is continuous, the constants
\[
\nu_-
=
\min_{\boldsymbol\theta\in I_{M}}
g(\boldsymbol\theta)>0,
\qquad
\nu_+
=
\max_{\boldsymbol\theta\in I_{M}}
g(\boldsymbol\theta)<\infty
\]
are well defined.

Using the reconstruction formula \eqref{eq:reconstruction_formula}, we get the bound
\[
\left(
\frac{m_0}{\nu_+}
\right)^{1/k}
\leq
u_1^\epsilon
\leq
\left(
\frac{M_0}{\nu_-}
\right)^{1/k}.
\]
Using \eqref{eq:global_ratio_bounds} and the reconstruction formula \eqref{eq:reconstruction_formula}, we find
\[
\frac{m}{M}
\left(
\frac{m_0}{\nu_+}
\right)^{1/k}
\leq
u_i^\epsilon
\leq
\frac{M}{m}
\left(
\frac{M_0}{\nu_-}
\right)^{1/k},
\qquad i=2,\ldots,n.
\]
Thus, with
\[
m_*
=
\frac{m}{M}
\left(
\frac{m_0}{\nu_+}
\right)^{1/k}, \quad M_*
=
\frac{M}{m}
\left(
\frac{M_0}{\nu_-}
\right)^{1/k},
\]
we obtain
\[
\dfrac{M m_*}{m}\leq u_1^\eps\leq \dfrac{mM_*}{M},\, \quad \, \quad m_*
\leq
u_i^\epsilon(x,t)
\leq
M_*,
\qquad i=2,\ldots,n.
\]

At $t=0$, homogeneity of $V(\mathbf{U})$ gives
\[
V^\epsilon_0
=
(u^\epsilon_{1,0})^k
g(\boldsymbol\theta^\epsilon_0),
\]
so the reconstruction formula \eqref{eq:reconstruction_formula} yields
\[
\mathbf U^\epsilon(\cdot,0)
=
\mathbf U^\epsilon_0.
\]
The regularity of $V^\epsilon$, the ratios $\theta^\epsilon_i$, and the reconstruction
formula \eqref{eq:reconstruction_formula} then gives
\[
\mathbf U^\epsilon
\in
C\bigl([0,T];C_b^{1+\alpha}(\mathbb R^d;\mathcal{U})\bigr)
\cap
C^1\bigl([0,T];C_b^\alpha(\mathbb R^d;\mathcal{U})\bigr).
\]

By the equivalence established in
Lemma~\ref{lem:decoupled_system}, the reconstructed state
$\mathbf U^\epsilon$ solves
\eqref{eq:homogeneous_KK_viscos_2} and satisfies the initial data \eqref{eq:homogeneous_KK_viscos_initial_data}. Finally, uniqueness of the
scalar parabolic problem, uniqueness of the characteristic flow,
and the reconstruction formula imply uniqueness of
$\mathbf U^\epsilon$. This completes the proof.
\end{proof}
\section{The vanishing viscosity limit of the system \eqref{eq:homogeneous_KK_viscos_2}}\label{sec: vanishing} 
In this Section, we pass to the limit $\epsilon \rightarrow 0$ in the Cauchy problem \eqref{eq:homogeneous_KK_viscos_2}-\eqref{eq:homogeneous_KK_viscos_initial_data} and prove the well-posedness of the global weak solutions of the Cauchy problem \eqref{eq:homogeneous_KK}-\eqref{initial_Data}. 
In order to prove this, we first prove that the sequence $\{V^\eps\}_{\eps>0}$ and $\{\theta_i^\eps\}_{\eps>0}, \, i=2, 3, \ldots, n$ are precompact in $L^1_{loc}(\Omega_T)$. This then allows us to obtain a strong limit for $\mathbf{U}^\eps\rightarrow \mathbf{U}$ in $L^1_{loc}(\Omega_T)$. This strong limit, along with gradient estimates for $V^\eps$ (see Lemma \ref{lem:power_entropy_V}) allows us to pass to the limit in the system \eqref{eq:homogeneous_KK_viscos_2}, which provides us the weak solution of the Cauchy problem \eqref{eq:homogeneous_KK}-\eqref{initial_Data}.

In what follows, we first prove that the sequence $\{V^\epsilon\}_{\epsilon>0}$ is precompact in $L^1_{loc}(\Omega_T)$. We prove this in the following Lemma.
\begin{lemma}[$L^1_{loc}$ precompactness of the sequence $\{V^\epsilon\}_{\epsilon>0}$]
\label{lemma: L1_compact}
Let $k>0$, $T>0$, and $0<\epsilon\leq1$, and assume that the initial data $\mathbf{U}_0^\epsilon$ defined in \eqref{eq:homogeneous_KK_viscos_initial_data} satisfy
\begin{align*}
\mathbf U_0^\epsilon \in {[m,M]}^n \subset \mathcal{U},
\qquad 0<m<M.
\end{align*}
Define the constants $m_0$ and $M_0$ as in \eqref{eq:constantm_0M_0}, such that $V_0^\epsilon$ defined in \eqref{eq:V_initial_data} satisfy
\begin{align}\label{eq:initial_bound_V}
0<m_0\leq V_0^\epsilon\leq M_0.
\end{align}
Further, assume that
\begin{equation}
\label{eq:initial_variation_V}
\sup_{0<\epsilon\leq1}
\int_{\mathbb R^d}|\nabla V_0^\epsilon(x)|\,dx
\leq \Lambda_0
\end{equation}
for some uniform constant $\Lambda_0>0$, independent of $\epsilon$.

For each $\epsilon> 0$, let $V^\epsilon$ be a classical solution of the Cauchy problem for \eqref{eq:scalar_velocity_equation_main} and \eqref{eq:V_initial_data} satisfying \eqref{eq:initial_bound_V}--\eqref{eq:initial_variation_V}. Then the sequence $\{V^\epsilon\}_{\epsilon>0}$ is relatively compact in $L^1_{loc}(\Omega_T)$.
\end{lemma}
\begin{proof}
Recall the transformation
\[
W^\epsilon=(V^\epsilon)^{1/k}.
\]
Since $m_0\leq V_0^\epsilon\leq M_0$, in view of Proposition \ref{lem:global_viscous}, we have
\[
\underline w:=m_0^{1/k}
\leq W^\epsilon
\leq M_0^{1/k}=:\overline w.
\]
The evolution equation of $W^\epsilon$ is given by
\eqref{W_equation} and can be written in a compact form as
\begin{equation}
\label{eq:W_epsilon}
\partial_tW^\epsilon+\nabla\cdot F(W^\epsilon)
=
\epsilon\Delta B(W^\epsilon),
\end{equation}
with initial data
\begin{equation}\label{eq:initialdatafor_W}
    W^\epsilon(.,0) = W^\epsilon_0 := (V_0^\epsilon)^{1/k}.
\end{equation}
In \eqref{eq:W_epsilon}, we denote
\[
F(W) =(f_1,\ldots,f_d):=W^{k+1}(1,\ldots,1), \quad 
B(W)=\frac{k}{k+1}W^{k+1}.
\]
In particular, $B'(W)=kW^k$, and therefore
\begin{equation}
\label{eq:Bprime_bounds}
km_0\leq B'(W^\epsilon)\leq kM_0.
\end{equation}
We now prove precompactness of $W^\epsilon$. Since
$w\mapsto w^k$ is Lipschitz continuous on
$[\underline w,\overline w]$, this implies the corresponding
compactness of $V^\epsilon$. Note that $W^\epsilon$ satisfies a
classical scalar viscous conservation law with nonlinear viscosity of the form \eqref{eq:W_epsilon}. Therefore, we can
obtain space-time estimates for $W^\eps$ following Kruzhkov \cite{Kruzkov} and
Alibaud et al. \cite{alibaud_entropy}.

Let us define
\[
\vartheta(x)=\exp\left(-\sqrt{1+|x|^2}\right), \quad x\in \mathbb{R}^d.
\]
Then
\begin{align}
\label{vartheta_bounds}
\vartheta\in L^1(\mathbb R^d),
\qquad
0<\vartheta\leq1,
\qquad
|\nabla\vartheta|\leq\vartheta,
\qquad
|\Delta\vartheta|\leq C\vartheta.
\end{align}
Fix $y\in\mathbb R^d$ and define
\[
U(x,t):=W^\epsilon(x+y,t),
\qquad
Z(x,t):=W^\epsilon(x,t),
\qquad
w:=U-Z.
\]
Since \eqref{eq:W_epsilon} is invariant under spatial translations,
we obtain
\begin{equation}
\label{W_updated}
w_t+\nabla\cdot(F(U)-F(Z))
=
\epsilon\Delta(B(U)-B(Z)).
\end{equation}
Define
\[
b(x,t)=\int_0^1 DF(Z+sw)\,ds,
\qquad
a(x,t)=\int_0^1B'(Z+sw)\,ds.
\]
Then, we have
\[
F(U)-F(Z)=bw,
\qquad
B(U)-B(Z)=aw,
\]
and hence \eqref{W_updated} becomes
\begin{equation}
\label{eq:difference_W}
w_t+\nabla\cdot(bw)
=
\epsilon\Delta(aw).
\end{equation}
Moreover,  \eqref{eq:Bprime_bounds} and Proposition \ref{lem:global_viscous} implies that there exist a constant $L_F>0$ such that
\begin{align}
\label{a_b_bounds}
|b|
\leq
\sup_{\underline w\leq \tau\leq\overline w}|DF(\tau)|
=:L_F,
\qquad
km_0\leq a\leq kM_0.
\end{align}
Then by multiplying with a sequence of convex functions converging
to $\sgn(w)$ and in view of the monotonicity of $B'$, it is easy to
obtain the following inequality in the sense of distributions (see, e.g., \cite{carrillo1999entropy, chen2003well}).
\begin{equation}
\label{eq:absolute_value_W}
\partial_t|w|
+
\nabla\cdot(b|w|)
\leq
\epsilon\Delta(a|w|).
\end{equation}
Testing \eqref{eq:absolute_value_W} with $\vartheta$ and using a
cutoff function argument, one can rigorously obtain
\begin{align}
\label{eq:absolute_value_W_integral}
\int_{\mathbb R^d}|w(x,t)|\vartheta(x)\,dx
&\leq
\int_{\mathbb R^d}|w(x,0)|\vartheta(x)\,dx
\nonumber\\
&\quad+
\int_0^t\int_{\mathbb R^d}
(b|w|)\cdot\nabla\vartheta\,dx\,ds
\nonumber\\
&\quad+
\int_0^t\int_{\mathbb R^d}
\epsilon(a|w|)\Delta\vartheta\,dx\,ds.
\end{align}
Noting the bounds
\eqref{vartheta_bounds} and \eqref{a_b_bounds}, we then have
\begin{equation}
\label{eq:absolute_value_W_integral_updated}
\int_{\mathbb R^d}|w(x,t)|\vartheta(x)\,dx
\leq
\int_{\mathbb R^d}|w(x,0)|\vartheta(x)\,dx
+
C\int_0^t\int_{\mathbb R^d}|w|\vartheta\,dx\,ds,
\end{equation}
where $C:=C(m_0,M_0)$ is a uniform constant and is independent of
$\epsilon$. Therefore, using Gronwall's inequality, we have
\begin{align}
\label{space_continuity_with_vartheta}
\int_{\mathbb R^d}
|W^\epsilon(x+y,t)-W^\epsilon(x,t)|\vartheta(x)\,dx
\leq
e^{CT}
\int_{\mathbb R^d}
|W^\epsilon(x+y,0)-W^\epsilon(x,0)|\vartheta(x)\,dx.
\end{align}

We now use \eqref{space_continuity_with_vartheta} to obtain uniform
bounded variation estimates. Let $y=he_i$, where $h>0$ and
$e_i$ is the $i$-th standard basis vector of $\mathbb{R}^d$. Dividing
\eqref{space_continuity_with_vartheta} by $h$, we obtain
\begin{align}
\label{space_continuity_h_division}
\int_{\mathbb R^d}
\frac{
|W^\epsilon(x+he_i,t)-W^\epsilon(x,t)|
}{h}
\vartheta(x)\,dx
\leq
e^{CT}
\underbrace{\int_{\mathbb R^d}
\frac{
|W_0^\epsilon(x+he_i)-W_0^\epsilon(x)|
}{h}
\vartheta(x)\,dx}_{=: \it I}.
\end{align}
Since $\vartheta\leq1$, it follows that
\begin{align}
 \it I=\int_{\mathbb R^d}
\frac{
|W_0^\epsilon(x+he_i)-W_0^\epsilon(x)|
}{h}
\vartheta(x)\,dx
&\leq
\int_{\mathbb R^d}
\frac{
|W_0^\epsilon(x+he_i)-W_0^\epsilon(x)|
}{h}\,dx
\nonumber\\
&\leq
\int_{\mathbb R^d}
|\partial_{x_i}W_0^\epsilon(x)|\,dx.
\label{eq:initial_difference_BV}
\end{align}
Moreover, since
\[
\nabla W_0^\epsilon
=
\frac{1}{k}(V_0^\epsilon)^{1/k-1}
\nabla V_0^\epsilon,
\]
and $m_0\leq V_0^\epsilon\leq M_0$. The Assumption
\eqref{eq:initial_variation_V} therefore gives
\begin{equation}
\label{eq:initial_variation_W}
\sup_{0<\epsilon\leq1}
\int_{\mathbb R^d}|\nabla W_0^\epsilon(x)|\,dx
\leq C\Lambda_0.
\end{equation}

Taking the limit $h\to0$ in
\eqref{space_continuity_h_division}, we obtain
\begin{align}
\label{space_continuity_derivative_bound}
\int_{\mathbb R^d}
|\partial_{x_i}W^\epsilon(x,t)|
\vartheta(x)\,dx
\leq
e^{CT}
\int_{\mathbb R^d}
|\partial_{x_i}W_0^\epsilon(x)|\,dx.
\end{align}
Summing over $i=1,\ldots,d$ and using
\eqref{eq:initial_variation_W}, we obtain
\begin{align}
\label{space_continuity_global_BV}
\sup_{0\leq t\leq T}
\int_{\mathbb R^d}
|\nabla W^\epsilon(x,t)|\vartheta(x)\,dx
\leq C e^{CT}\Lambda_0.
\end{align}

For any fixed radius $R>0$, we have
\[
\vartheta(x)\geq e^{-\sqrt{1+R^2}}
\qquad\text{for }x\in\mathcal B_R.
\]
This allows us to localize the estimate \eqref{space_continuity_global_BV}, which gives
\begin{align}
\label{space_continuity_local_BV}
\sup_{0\leq t\leq T}
\int_{\mathcal B_R}
|\nabla W^\epsilon(x,t)|\,dx
\leq C_R,
\end{align}
where $C_R$ is uniform in $\epsilon$. Since $W^\epsilon$ is also
uniformly bounded, it follows that the sequence
$\{W^\epsilon\}_{\epsilon>0}$ is uniformly bounded in
\[
L^\infty(0,T;BV_{loc}(\mathbb R^d)).
\]

To establish uniform equicontinuity of $\{W^\eps\}_{\eps>0}$ in time, let
$\rho\in C_c^\infty(\mathbb R^d)$ be a standard nonnegative
mollifier supported in $\mathcal B_R$ and set
\[
\rho_h(x)=h^{-d}\rho(x/h),
\qquad
W_h^\epsilon=\rho_h*W^\epsilon.
\]
For every $t\in[0,T]$, we have 
\begin{align}
{\|W_h^\epsilon(t)-W^\epsilon(t)\|}_{L^1(\mathcal B_R)}
&\leq
\int_{\mathbb R^d}\rho_h(z)
\int_{\mathcal B_R}
|W^\epsilon(x-z,t)-W^\epsilon(x,t)|\,dx\,dz.
\label{eq:mollification_space_error}
\end{align}
Using the uniform $BV_{loc}$ estimate
\eqref{space_continuity_local_BV}, we obtain, for $0<h\leq1$,
\begin{equation}
\label{eq:mollification_BV_error}
\sup_{0\leq t\leq T}
{\|W_h^\epsilon(t)-W^\epsilon(t)\|}_{L^1(\mathcal B_R)}
\leq C_Rh.
\end{equation}

On the other hand, convolving \eqref{eq:W_epsilon} with $\rho_h$ gives
\[
\partial_tW_h^\epsilon
=
-F(W^\epsilon)*\nabla\rho_h
+
\epsilon B(W^\epsilon)*\Delta\rho_h.
\]
Since $W^\epsilon$ is uniformly bounded, we have
\[
{\|F(W^\epsilon)\|}_{L^\infty}
+
{\|B(W^\epsilon)\|}_{L^\infty}
\leq C.
\]
Furthermore,
\[
{\|\nabla\rho_h\|}_{L^1}
\leq Ch^{-1},
\qquad
{\|\Delta\rho_h\|}_{L^1}
\leq Ch^{-2}.
\]
Therefore, we obtain
\begin{equation}
\label{eq:mollified_time_derivative}
{\|\partial_tW_h^\epsilon(t)\|}_{L^\infty(\mathbb R^d)}
\leq
C(h^{-1}+h^{-2}),
\end{equation}
uniformly in $\epsilon$.

Let $0\leq t\leq t+\tau\leq T$. Then
\[
{\|W_h^\epsilon(t+\tau)-W_h^\epsilon(t)\|}_{L^1(\mathcal B_R)}
\leq
C_R\tau(h^{-1}+h^{-2}).
\]
Using \eqref{eq:mollification_BV_error} at times $t$ and
$t+\tau$, we therefore obtain
\begin{align}
{\|W^\epsilon(t+\tau)-W^\epsilon(t)\|}_{L^1(\mathcal B_R)}
&\leq
2C_Rh
+
C_R\tau(h^{-1}+h^{-2}).
\label{eq:time_translation_interpolation}
\end{align}
Choosing
\[
h=\tau^{1/3},
\]
we conclude that
\begin{equation}
\label{eq:local_time_translation_W}
\lim_{\tau\to0}
\sup_{0<\epsilon\leq1}
\sup_{0\leq t\leq T-\tau}
\int_{\mathcal B_R}
|W^\epsilon(x,t+\tau)-W^\epsilon(x,t)|\,dx
=0.
\end{equation}

Since the sequence $\{W^\epsilon\}_{\epsilon>0}$ is uniformly
bounded in
$L^\infty(0,T;BV_{loc}(\mathbb R^d))$, the compact
embedding $BV(\mathcal B_R)\Subset L^1(\mathcal B_R)$ provides compactness in space. Together with the uniform time
equicontinuity \eqref{eq:local_time_translation_W}, the
Arzel\`a--Ascoli theorem and the compactness theorem for $BV$
functions imply that the sequence
$\{W^\epsilon\}_{\epsilon>0}$ is precompact in
\[
L^1_{loc}(\mathbb R^d\times[0,T]).
\]
Using the Lipschitz continuity of the map $s\mapsto s^k$ on
$[\underline w,\overline w]$, the precompactness of the sequence
$\{V^\epsilon\}_{\epsilon>0}$ in
$L^1_{loc}(\mathbb R^d\times[0,T])$ is straightforward.
\end{proof}
Although Lemma \ref{lemma: L1_compact} provides strong compactness for the sequence $\{V^\eps\}_{\eps>0}$, it is still not enough to pass to the limit in the nonlinear viscous term of the viscous system \eqref{eq:homogeneous_KK_viscos_2}. In order to pass to the limit in the system \eqref{eq:homogeneous_KK_viscos_2}, we need a strong limit for the sequence $\{\boldsymbol\theta^\eps\}_{\eps>0}$ as $\epsilon \rightarrow 0$ and also gradient control for the sequence $\{\nabla V^\epsilon\}_{\epsilon>0}$. We prove the gradient estimates of $V^\eps$ in the following lemma. 
\begin{lemma}[Local viscous gradient estimate for $\{V^\eps\}_{\eps>0}$]
\label{lem:power_entropy_V}
Let $k>0$, $\beta>0$, and $T>0$. Assume that
\[
V^\epsilon\in C^{2,1}(\mathbb R^d\times[0,T])
\]
be a classical solution of \eqref{eq:scalar_velocity_equation_main} and \eqref{eq:V_initial_data} such that
\[
0<m_0\leq V^\epsilon(x,t)\leq M_0,
\quad
\text{for all }(x,t)\in\mathbb R^d\times[0,T].
\]
Define
\[
{\eta}_\beta(V)=V^{-\beta},
\]
and 
\begin{align*}
q_\beta(V)
=\begin{cases}
\frac{\beta(k+1)}{\beta-1}V^{1-\beta},\quad \beta\neq 1,\\[1ex]
-(k+1)\log V,\quad \beta=1.
\end{cases}
\end{align*}
Then $V^\epsilon$ satisfies the pointwise identity
\begin{align}
\partial_t\eta_\beta(V^\epsilon)
+
\sum_{j=1}^d
\partial_{x_j}q_\beta(V^\epsilon)
&=
-\epsilon\beta k
\nabla\cdot
\left(
(V^\epsilon)^{-\beta}\nabla V^\epsilon
\right)-\epsilon\beta(1+k\beta)
(V^\epsilon)^{-\beta-1}
|\nabla V^\epsilon|^2.
\label{eq:power_entropy_identity}
\end{align}
In particular, for every compact set $K\Subset\mathbb R^d$,
there exists a constant $C_{K,T}>0$, independent of $\epsilon$, such that
\begin{equation}
\label{eq:local_gradient_estimate}
\epsilon
\int_0^T\int_K
|\nabla V^\epsilon|^2\,dx\,dt
\leq C_{K,T}.
\end{equation}
Equivalently,
\[
\left\{
\sqrt{\epsilon}\,\nabla V^\epsilon
\right\}_{\epsilon>0}
\quad\text{is bounded in }
L^2_{loc}
(\mathbb R^d\times(0,T)).
\]
\end{lemma}

\begin{proof}
Since
\[
\eta_\beta'(V)=-\beta V^{-\beta-1},
\]
multiplying \eqref{eq:scalar_velocity_equation_main} by
$-\beta(V^\epsilon)^{-\beta-1}$ gives
\begin{align}
\partial_t(V^\epsilon)^{-\beta}
-
\beta(k+1)(V^\epsilon)^{-\beta}
\sum_{j=1}^d\partial_{x_j}V^\epsilon
&=
-\epsilon\beta k
(V^\epsilon)^{-\beta}\Delta V^\epsilon
-\epsilon\beta
(V^\epsilon)^{-\beta-1}
|\nabla V^\epsilon|^2.
\label{eq:entropy_multiplied}
\end{align}
For $\beta\neq1$,
\[
q_\beta'(V)
=
-\beta(k+1)V^{-\beta},
\]
and the same relation holds for $\beta=1$ with
$q_1(V)=-(k+1)\log V$. Hence
\[
-\beta(k+1)(V^\epsilon)^{-\beta}
\sum_{j=1}^d\partial_{x_j}V^\epsilon
=
\sum_{j=1}^d
\partial_{x_j}q_\beta(V^\epsilon).
\]
Further, we use the identity
\[
\nabla\cdot
\left(
(V^\epsilon)^{-\beta}\nabla V^\epsilon
\right)
=
(V^\epsilon)^{-\beta}\Delta V^\epsilon
-
\beta
(V^\epsilon)^{-\beta-1}
|\nabla V^\epsilon|^2.
\]
Consequently,
\begin{align}\label{eq: V_laplacian}
-\beta k(V^\epsilon)^{-\beta}\Delta V^\epsilon
&=
-\beta k
\nabla\cdot
\left(
(V^\epsilon)^{-\beta}\nabla V^\epsilon
\right)
-\beta^2k
(V^\epsilon)^{-\beta-1}
|\nabla V^\epsilon|^2.
\end{align}
Substituting \eqref{eq: V_laplacian} in \eqref{eq:entropy_multiplied} yields
\eqref{eq:power_entropy_identity}.

It remains to obtain the local gradient estimate \eqref{eq:local_gradient_estimate}. Let
$K\Subset\mathbb R^d$ and $\chi\in C_c^\infty(\mathbb R^d),
\, \,
0\leq\chi\leq1$ be a standard mollifier. Multiplying \eqref{eq:power_entropy_identity} by $\chi^2$ and
integrating over $\mathbb R^d\times(0,T)$ gives
\begin{align}
\epsilon\beta(1+k\beta)
&\int_0^T\int_{\mathbb R^d}
\chi^2(V^\epsilon)^{-\beta-1}
|\nabla V^\epsilon|^2\,dx\,dt
\nonumber\\
&=
\int_{\mathbb R^d}
\left(
(V^\epsilon_0)^{-\beta}
-
(V^\epsilon(T))^{-\beta}
\right)\chi^2\,dx
+
\int_0^T\int_{\mathbb R^d}
q_\beta(V^\epsilon)
\sum_{j=1}^d\partial_{x_j}(\chi^2)\,dx\,dt\nonumber\\
&+
\epsilon\beta k
\int_0^T\int_{\mathbb R^d}
(V^\epsilon)^{-\beta}
\nabla V^\epsilon\cdot\nabla(\chi^2)\,dx\,dt.
\label{eq:localized_entropy}
\end{align}
Since $m_0\leq V^\epsilon\leq M_0$, both
$(V^\epsilon)^{-\beta}$ and $q_\beta(V^\epsilon)$ are uniformly
bounded. Thus the first two terms on the right-hand side of
\eqref{eq:localized_entropy} are bounded independently of $\epsilon$.

For the last term of \eqref{eq:localized_entropy}, Young's inequality and
$\nabla(\chi^2)=2\chi\nabla\chi$ give
\begin{align*}
&2\epsilon\beta k
\int_0^T\int_{\mathbb R^d}
\chi(V^\epsilon)^{-\beta}
|\nabla V^\epsilon||\nabla\chi|\,dx\,dt
\\
&\qquad\leq
\frac{\epsilon\beta(1+k\beta)}{2}
\int_0^T\int_{\mathbb R^d}
\chi^2(V^\epsilon)^{-\beta-1}
|\nabla V^\epsilon|^2\,dx\,dt
+C\epsilon.
\end{align*}
Absorbing the first term
into the left-hand side of \eqref{eq:localized_entropy}, and using
\[
(V^\epsilon)^{-\beta-1}\geq M_0^{-\beta-1},
\]
we conclude that
\[
\epsilon
\int_0^T\int_K
|\nabla V^\epsilon|^2\,dx\,dt
\leq C_{K,T}.
\]
This proves the Lemma.
\end{proof}
In view of the precompactness of $V^\eps$ in $L^1_{loc}(\mathbb{R}^d\times (0, T))$ and the local gradient control of $V^\epsilon$, we now prove the following global well-posedness result for the Cauchy problem \eqref{eq:homogeneous_KK}-\eqref{initial_Data} as the first main result of this article.

\begin{theorem}[Well-posedness of the global weak solutions of the Cauchy prolem \eqref{eq:homogeneous_KK}-\eqref{initial_Data}]\label{main_theorem_1}
Let $k>0$, $T>0$ and $M > m>0 $ be constants such that the initial datum 
$\mathbf{U}_0\in L^\infty(\mathbb{R}^d; \mathbb{R}^n)\cap BV_{loc}(\mathbb{R}^d; \mathbb{R}^n)$ satisfies
\begin{align}\label{initial_invariant}
\mathbf{U}_0\in [m, M]^n \subset {\mathcal U}. 
\end{align}
Further, let there is a constant $\Lambda_0>0$ such that
\begin{equation}
\label{eq:initial_variation_V_main}
\int_{\mathbb R^d}|\nabla V_0(x)|\,dx
\leq \Lambda_0.
\end{equation}
\\
Assume that $\{\mathbf{U}^\epsilon\}_{\epsilon>0}$ is a sequence of smooth solutions of the Cauchy problem \eqref{eq:homogeneous_KK_viscos_2}-\eqref{eq:homogeneous_KK_viscos_initial_data}. Then, there exists a subsequence of ${\{\mathbf{U}^\epsilon}\}_{\epsilon>0}$ (still labelled ${\{\mathbf{U}^\epsilon}\}_{\epsilon>0}$) and 
a function $\mathbf{U} = (u_1,u_2, \ldots, u_n)^T\in L^\infty(\Omega_T; \mathcal{U})\cap BV_{loc}(\Omega_T; \mathcal{U})$, such that $\mathbf{U}^\epsilon \rightarrow   \mathbf{U}$ holds a.e.\ for  $\epsilon\rightarrow 0$.\\
The limit $\mathbf{U}$ is a unique weak solution of the Cauchy problem for the hyperbolic system \eqref{eq:homogeneous_KK} in the sense of Definition \ref{weak_soln}. 
\end{theorem}
\begin{proof}
In what follows, we prove that $\mathbf{U}^\eps\rightarrow \mathbf{U}$ in $(L^p_{loc}(\Omega_T))^n, \, \, p\in [1, \infty)$. First, in view of Lemma \ref{lemma: L1_compact}, there exists a subsequence of the sequence $\{V^\epsilon\}_{\epsilon>0}$ (still denoted by $\{V^\epsilon\}_{\epsilon>0}$) and a function $V\in L^\infty(\Omega_T)\cap BV_{loc}(\Omega_T)$ such that $V^\epsilon\rightarrow V$ in $L^1_{loc}(\Omega_T)$, as $\epsilon \rightarrow 0$.

To obtain the convergence of the sequence $\{\mathbf{U}^\eps\}_{\eps>0}$, it remains to establish that $\theta_i^\eps \to \theta_i$ in $L^1_{loc}(\Omega_T)$ for each $i=2, 3, \ldots, n$, where $\theta_i^\eps$ satisfies \eqref{eq:ratio_transport_equation_main}. Since $V^\eps \to V$ in $L^1_{loc}(\Omega_T)$ and the sequence $\{V^\eps\}_{\eps>0}$ is uniformly bounded with respect to $\eps$, the dominated convergence theorem implies that $V^\eps=(W^\eps)^k \to V$ in $L^p_{loc}(\Omega_T)$ for all $1\leq p<\infty$. Furthermore, the local gradient estimates obtained in Lemma \ref{lem:power_entropy_V} show that $\eps\nabla V^\eps \to 0$ in $L^2_{loc}(\Omega_T)$. Therefore, the transport velocities $\mathbf{c}^\eps=V^\eps \mathbf{e}-\epsilon \nabla V^\eps, \, \,\mathbf{e}=(1, 1, \ldots, 1)^\top$, converge strongly in $L^1_{loc}$. In particular, we have
\[
\mathbf{c}^\epsilon \to \mathbf{c} := V\mathbf{e} \, \, \mathrm{in}\,\,L^1_{loc}(\Omega_T).
\]
Note that the evolution equation for $W^\epsilon=(V^\epsilon)^{1/k}$ given by \eqref{W_equation} can also be expressed as a transport equation of the form
\begin{equation}\label{W_continuity}
W^\eps_t+\nabla\cdot (W^\eps\mathbf{c}^\epsilon)=0,
\end{equation}
Multiplying the transport equation \eqref{eq:ratio_transport_equation_main} by $W^\epsilon$ and using the equation \eqref{W_continuity} yields the following conservative equation for $W^\eps \theta_i^\eps$.
\begin{equation}
\label{eq:weighted_xi_epsilon}
\partial_t(W^\epsilon\theta_i^\epsilon) + \nabla \cdot(W^\epsilon\mathbf{c}^\epsilon\theta_i^\epsilon) = 0, \quad i=2, 3, \ldots, n.
\end{equation}

Thus, we consider the Cauchy problem 
\begin{align}
\label{eq:weighted_xi_cauchy}
\begin{cases}
\partial_t(W^\epsilon\theta_i^\epsilon) + \nabla \cdot(W^\epsilon\mathbf{c}^\eps\theta_i^\eps) = 0, \quad i=2, 3, \ldots, n\\
(W^\eps\theta_i^\eps)(x, 0)=W_0^\eps\theta_{i,0}^\eps,
\end{cases}
\end{align}
where $W_0^\epsilon=(V_0^\epsilon)^{1/k},$ and $V_0^\eps$ and $\theta_0^\eps$ are defined in \eqref{eq:V_initial_data}, \eqref{eq:xi_initial_data}, respectively.

In view of Proposition \ref{lem:global_viscous}, the sequence $\{\theta_i^\epsilon\}_{\eps>0}$ is uniformly bounded and thus, by the Banach-Alaoglu theorem, we can extract a subsequence (still denoted by $\{\theta_i^\eps\}_{\eps>0}$) and a function $\theta_i\in L^\infty(\Omega_T)$ such that $\theta_i^\epsilon \rightharpoonup^\ast \theta_i$ in $L^\infty(\Omega_T)$. The strong convergence of $V^\epsilon$ and $\mathbf{c}^\epsilon$ allows us to pass to the limit in \eqref{eq:weighted_xi_cauchy} in the sense of distributions, which implies that $\theta_i$ is a bounded weak solution of the following limit Cauchy problem. 
\begin{align}
\label{eq:weighted_xi_cauchy_limit}
\begin{cases}
\partial_t(W\theta_i) + \nabla \cdot(W\mathbf{c}\theta_i) = 0, \quad i=2, 3, \ldots, n\\
(W\theta_i)(x, 0)=W_0\theta_{i,0},
\end{cases}
\end{align}
Since $W, \mathbf{c} \in BV_{loc}(\Omega_T)$, the product $W \mathbf{c}\in BV_{loc}(\Omega_T)$. Thus, the argument in Lemma 2.8 and Proposition 2.10 of Ambrosio--Bouchut--De Lellis \cite{ambrosio2005well} applies to \eqref{eq:weighted_xi_cauchy_limit}, guaranteeing that $\theta_i^\epsilon\rightarrow \theta_i$ in $L^1_{loc}(\Omega_T)$ uniquely.  In view of Proposition \ref{lem:global_viscous} and the dominated convergence theorem, we can then obtain strong convergence in $L^p_{loc}(\Omega_T)$ for any $p\in [1, \infty)$.

Due to the strong convergence of the sequence $(V^\epsilon, \theta_2^\epsilon, \theta_3^\epsilon, \ldots, \theta_n^\epsilon)\to (V, \theta_2, \theta_3, \ldots, \theta_n)$ and the invertible relation  \eqref{eq:reconstruction_formula}, we can identify a subsequence $\{\mathbf{U}^\eps\}_{\eps>0}$ and
a unique function $\mathbf{U}= (u_1, u_2, \ldots, u_n)^\top$ in $L^\infty(\Omega_T; \mathcal{U})\cap BV_{loc}(\Omega_T; \mathcal{U})$ such that 
\begin{equation}\label{stronglimit}
\mathbf{U}^\eps \to \mathbf{U} \text{ a.e.\ in } (L^1_{loc}(\Omega_T))^n \text{ as } \eps \to 0. 
\end{equation}
We now proceed to prove that the limit $\mathbf{U}$ of the sequence of the approximate solutions $\mathbf{U}^\eps = (u_1^\eps, u_2^\eps, \ldots, u_n^\eps)^\top$ is the weak solution of the Cauchy problem \eqref{eq:homogeneous_KK}-\eqref{initial_Data}. Consider a sequence of smooth solutions $(u_1^\eps, u_2^\eps, \ldots, u_n^\eps)^\top$ of the Cauchy problem \eqref{eq:homogeneous_KK_viscos_2}-\eqref{eq:homogeneous_KK_viscos_initial_data} obtained using the convergent sequences $\{V^\eps\}_{\eps>0}$ and $\{\theta_i^\eps\}_{\eps>0}$ and the transformation \eqref{eq:reconstruction_formula}. Further, consider an arbitrary  vector-valued test function $\bm{\varphi}\in C_0^1(\Omega_T; \mathbb{R}^n)$. Then, by multiplying the system \eqref{eq:homogeneous_KK_viscos_2} with $\bm{\varphi}$ and integrating over $\Omega_T$, we obtain
\begin{equation}\label{weak_form_smooth}
    \displaystyle\iint_{\Omega_T} 
    \Big( \mathbf{U}^\eps\cdot {\bm \varphi}_t
          + \mathbf{U}^\eps V(\mathbf{U}^\eps)\cdot {\bm \varphi}_x
         \Big)\,dx\,dt
    + \displaystyle\int_{\mathbb{R}} \mathbf{U}_0^\eps\cdot {\bm \varphi}(\cdot, 0)\,dx = \epsilon \displaystyle\iint_{\Omega_T}  \left(\mathbf{U}^\eps \nabla V(\mathbf{U}^\eps)\right)\cdot \bm{\varphi}_x\, dx\, dt,
\end{equation}
From the strong limit $\mathbf{U}^\eps\rightarrow \mathbf{U}$ in \eqref{stronglimit}, and continuity of $V$, we deduce  $V\mathbf{(U^\eps)}\rightarrow V(\mathbf{U})$ a.e. in $\Omega_T$. Moreover, in view of local gradient estimate \eqref{eq:local_gradient_estimate}, it follows that
\[
{\|\epsilon \nabla V^\eps\|}_{L^2(\Omega_T)}^2
=
\epsilon^2{\|\nabla V^\eps\|}_{L^2(\Omega_T)}^2
\le \tilde C\epsilon \to 0.
\]
Hence, using the uniform bounds of $\mathbf{U}^\eps$ from the Proposition \ref{lem:global_viscous} and applying Cauchy-Schwarz inequality, we obtain
\[
\epsilon \displaystyle\iint_{\Omega_T}  \left(\mathbf{U}^\eps\nabla V(\mathbf{U}^\eps)
\right)\cdot \bm{\varphi}_x\, dx\, dt \to 0
\qquad\text{as }\epsilon\to0.
\]
Passing to the limit in \eqref{weak_form_smooth} and using the strong convergence of $\mathbf U^\epsilon\to \mathbf U$ and $\mathbf{U}_0^\eps\rightarrow \mathbf{U}_0$, we conclude that the limit function  $\mathbf{U}=(u_1, u_2, \ldots, u_n)^\top$ is the unique  weak solution of the Cauchy problem \eqref{eq:homogeneous_KK}-\eqref{initial_Data} in the sense of Definition \ref{weak_soln}. This completes the proof of the Theorem.
\end{proof}
\section{Stability of the viscous shock for the system \eqref{eq:homogeneous_KK_viscos_2}}\label{sec: viscous_shock}
So far, we have utilized the viscous approximation \eqref{eq:homogeneous_KK_viscos_2} as a regularization for the inviscid system \eqref{eq:homogeneous_KK} and to develop global weak solutions of the Cauchy problem \eqref{eq:homogeneous_KK}-\eqref{initial_Data}.  In Section \ref{sec: viscous}, we established the global well-posedness and $L^\infty$ bound of the smooth viscous solution of the Cauchy problem \eqref{eq:homogeneous_KK_viscos_2}-\eqref{eq:homogeneous_KK_viscos_initial_data} for the initial data belonging to the state space $\mathcal{U}$. In this Section, we analyze the existence and decay properties of the perturbation of the viscous shock with the help of the obtained mathematical structure \eqref{eq:scalar_velocity_equation_main}-\eqref{eq:ratio_transport_equation_main}. 

One of the primary motivations for studying the decay behavior of large perturbations of shocks is to understand the stability of the admissible shock waves of the inviscid hyperbolic system \eqref{eq:homogeneous_KK}. Although the analysis of a general multi-dimensional $n\times n$ system is challenging, we can still derive the time-asymptotic stability of the viscous shock profile of the system \eqref{eq:homogeneous_KK_viscos_2} under certain restrictive assumptions on the state space (see \eqref{eq:viscousassumption} below). In order to achieve the $L^2$-time decay estimate, we use the relative entropy method up to a dynamical shift $X(t)$, which was also utilized in \cite{MR4195742, MR3592682} (see \cite{dafermos2005hyperbolic} for more details on the relative entropy method). However, the presence of nonlinear viscous terms requires careful handling of the resulting residuals in the relative entropy estimates. In order to control these residual terms, we utilize the a-priori estimates of the viscous system \eqref{eq:homogeneous_KK_viscos_2} obtained in Proposition \ref{lem:global_viscous}.

In what follows, we restrict our attention to a subset $\widetilde{\mathcal{U}}$ of the state space $\mathcal{U}$ defined by
\begin{align}\label{eq:viscousassumption}
     \widetilde{\mathcal{U}} =\left\{\left(u_{1}^\eps, \ldots, u_{n}^\eps\right)\in \mathcal{U} \Big|\,\,u_1^\eps\geq c>0, \quad\, \, \frac{u_{i}^\eps}{u_{1}^\eps} = \gamma_i, \quad \gamma_i\in \mathbb{R},\,  i = 2,\ldots,n\right\} \subset \mathcal{U}.
\end{align}
To prove the stability of the viscous shock of the Cauchy problem \eqref{eq:homogeneous_KK_viscos_2}-\eqref{eq:homogeneous_KK_viscos_initial_data}, we adopt a strategy similar to the one used in Section \ref{sec: viscous}, i.e., we first explore the viscous shock profile associated to the transformed variable $W^\epsilon = (V^\epsilon)^{\frac{1}{k}}$ and establish the time-asymptotic stability of the associated viscous shock $\widetilde{W}^\epsilon$ (see also Remark \ref{rem:forviscousshock}). Since the ratio $\theta^\epsilon _i,\, \, i=2, 3, \ldots, n$, is constant in the state space $\widetilde{\mathcal{U}}$, the transport structure in \eqref{eq:ratio_transport_equation_main} implies that $\theta^\epsilon_i$ remains constant for all $(x, t)\in \Omega\times (0, T)$. This allows us to return to the original variable $\mathbf{U}^\epsilon$ by using the inverse transformation \eqref{eq:reconstruction_formula}.

First, we recall the Cauchy problem of $W^\epsilon$ from \eqref{eq:W_epsilon}-\eqref{eq:initialdatafor_W}. We further assume that there exist constant states $W_\pm\in \widetilde{\mathcal{U}}$ such that $W^\epsilon_0(x)\longrightarrow W_\pm \quad \text{as}~x \to \pm \infty$. Then the Cauchy problem \eqref{eq:W_epsilon}-\eqref{eq:initialdatafor_W} reduces to
\begin{align}
\partial_tW^\epsilon&+\nabla\cdot F(W^\epsilon)
=
\epsilon \Delta B(W^\epsilon),\label{eq:equationforW}
\end{align}
with initial data 
\begin{align}\label{eq:initialdataforW}
    W^\epsilon(x,0) = W^\epsilon_0(x)\longrightarrow W_\pm \quad \text{as}~x \to \pm \infty.
\end{align}
In \eqref{eq:equationforW}, $F(W)=(f_1, \ldots, f_d) = W^{k+1}(1,\ldots,1)$, $B(W) = \frac{k}{k+1}W^{k+1}$, $x = (x_1, x')\in \Omega_s: =\mathbb{R}\times \mathbb{T}^{d-1}\subset \Omega$ and $\mathbb{T}^{d-1}:= \mathbb{R}^{d-1}/\mathbb{Z}^{d-1}$, $d\geq 2$ is $d-1$ dimensional flat torus.

Since $\theta^\epsilon_i$ for $i=2, 3, \ldots, n$ is constant for $\mathbf{U}^\eps\in \mathcal{U}$, using the inverse transformation \eqref{eq:reconstruction_formula}, we can find constant states $\mathbf{U}_\pm\in \widetilde{\mathcal{U}}$ defined by
\begin{align}\label{eq:farfieldstates}
    \mathbf{U}_\pm :=  \left(u_{1\pm}, u_{2\pm},\ldots,u_{n\pm}\right)^\top = \left(\frac{W_\pm}{\bar{g}^{1/k}}, \gamma_2\frac{W_\pm}{\bar{g}^{1/k}},\ldots,\gamma_n\frac{W_\pm}{\bar{g}^{1/k}} \right)^\top,
\end{align}
where the constant $\bar{g}= g(\gamma), \, \gamma = (\gamma_2,\gamma_3,\ldots, \gamma_n)^\top$. 

Before stating our second main result, we make the following assumption on the function $W^\eps$.
\begin{assumption}[Assumption on $W^\eps$ and far field states]\label{assumption2}
    For a given left far field state $W_-$, let $W^\eps$ satisfy $M_0{\|(W^\eps)^{k+1}\|}_{L^\infty}\leq m_0(W_-)^{k+1}$, where $M_0$ and $m_0$ are defined as in \eqref{eq:constantm_0M_0}. Further, assume that for any $\delta'\in(0,1)$, the far field states $W_\pm$, defined in \eqref{eq:initialdataforW}, satisfy $W_+ = W_- - \delta'$.  
\end{assumption}

With this setup in mind, we now state the second main result of this article, which concerns the $L^2$-stability and decay of viscous shock profiles of the Cauchy problem \eqref{eq:homogeneous_KK_viscos_2}-\eqref{eq:homogeneous_KK_viscos_initial_data} for $\mathbf{U}^\eps\in \widetilde{\mathcal U}$.
\begin{theorem}\label{main_theorem_2}
Consider the Cauchy problem \eqref{eq:homogeneous_KK_viscos_2}-\eqref{eq:homogeneous_KK_viscos_initial_data} with initial data $\mathbf{U}^\epsilon_0 \in \widetilde{\mathcal{U}}$. Let $ W_{-}$ be the far-field state as defined in \eqref{eq:initialdataforW} satisfying Assumption \ref{assumption2}. Further, let $\delta$ and $\delta_0$ be two uniform constants. 

For $0<\delta<\delta_0$ with $\delta_0\in(0,1)$, assume that $u_{i+} = u_{i-}-\delta, \, i=1, 2, \ldots, n$, where $u_{i_\pm}$ are defined as in \eqref{eq:farfieldstates} and $\delta$ denotes the shock strength. Let $\widetilde{\mathbf{U}}^\epsilon$ be the viscous shock of \eqref{eq:homogeneous_KK_viscos_2}-\eqref{eq:homogeneous_KK_viscos_initial_data} connecting the constant states $\mathbf{U}_-$ and $\mathbf{U}_+$ defined in \eqref{eq:farfieldstates}. Then for any smooth solution $\mathbf{U}^\epsilon$ of \eqref{eq:homogeneous_KK_viscos_2}--\eqref{eq:homogeneous_KK_viscos_initial_data} with initial data satisfying $\mathbf{U}^\epsilon_0-\widetilde{\mathbf{U}}^\epsilon\in (L^1\,\cap\, L^\infty)(\Omega_s)$, there exists an absolute continuous shift $X(t)$ in $\widetilde{\mathbf{U}}^\eps$ such that for all $t\in (0, T]$, we have the estimate
 \begin{align}\label{eq:decayforu_i}
     {\|u_i^\epsilon(\cdot, t)-\widetilde{u}_i^\epsilon(.-\sigma t-X(t))\|}_{L^2(\Omega_s)} = O(t^{-\frac{1}{4}}),\quad i = 1,\ldots, n,
 \end{align}
 where $\sigma$ is the speed of the inviscid shock corresponding to the genuinely nonlinear field of the inviscid system \eqref{eq:homogeneous_KK}.

 Moreover, for the dynamical shift $X(t)$, we have the estimate
 \begin{align}
     |\dot{X}(t)| = O(t^{-\frac{1}{4}}).
 \end{align}
\end{theorem}
In the following Sections, we prove Theorem \ref{main_theorem_2} by utilizing the mathematical structure \eqref{eq:scalar_velocity_equation_main}-\eqref{eq:ratio_transport_equation_main} of the viscous system \eqref{eq:homogeneous_KK_viscos_2}.
\subsection{Existence and time asymptotic decay of the viscous shock profile for \eqref{eq:equationforW}-\eqref{eq:initialdataforW} }\label{sectionforW}
In this Section, we prove the existence and time-asymptotic stability of the viscous shock profile for the Cauchy problem \eqref{eq:equationforW}-\eqref{eq:initialdataforW}. We prove the existence of the viscous shock profile in the following Lemma.
\begin{lemma}\label{lem:existenceofviscouswave}
    Let $W_->W_+$ be as defined in \eqref{eq:initialdataforW}. Suppose that the shock speed $\sigma$ is determined by the Rankine-Hugoniot relation $\sigma = \frac{f_1(W_-)-f_1(W_+)}{W_--W_+}$, where $f_1=W^{k+1}$. Then, for every $\epsilon>0$, there exists a planar viscous wave profile  $\widetilde{W}^\epsilon(\xi) = \widetilde{W}^\epsilon(x_1-\sigma t)$ of \eqref{eq:equationforW}-\eqref{eq:initialdataforW} satisfying the boundary value problem
    \begin{align}\label{eq:viscouseqforW}
    \begin{cases}
    -\sigma\partial_\xi\widetilde{W}^\epsilon + \partial_\xi f_1(\widetilde{W}^\epsilon) = \epsilon \partial_{\xi \xi}B(\widetilde{W}^\epsilon),\\
    \widetilde{W}^\epsilon|_{\pm \infty} = W_\pm.
    \end{cases}
    \end{align}
    Moreover, the viscous profile is monotonically decreasing and unique up to a translation. 
\end{lemma}
\begin{proof}
The proof of this Lemma follows directly from the classical theory for convex viscous scalar conservation laws; see \cite{ilin1960asymptotic} for more details.
\end{proof}
In what follows, we derive the contraction estimate and decay of large perturbations of the viscous shock for $W^\epsilon$, which is crucial to prove Theorem \ref{main_theorem_2}. The decay and $L^2$-contraction property for $W^\epsilon$ is established in Lemma \ref{lem:forW}. In order to prove Lemma \ref{lem:forW},  we need the following Poincaré-type inequality and a well-defined shift function $X(t)$.
\begin{lemma}[\cite{MR4195742}, Lemma 2.9]
   Suppose that $g:[0,1]\longrightarrow \mathbb{R}$ be a continuously differentiable function such that $\int_0^1z(1-z)|g'|^2\,dz<\infty$. Then the following holds
   \begin{align}\label{eq: poincare_type_inequality}
       \int_0^1\left|g-\int_0^1g \,dz\right|^2\,dz\leq \dfrac{1}{2}\int_0^1z(1-z)|g'|^2\,dz.
   \end{align}
\end{lemma}
Now we introduce $X(t)$ as a solution of the following initial value problem
\begin{align}\label{eq:shiftfunction}
\left\{
\begin{aligned}
    \dot{X}(t)
    &= -\dfrac{M}{\delta}\int_{\Omega_s}
    \left(W^\epsilon(x,t)-\widetilde{W}^\epsilon(x_1-\sigma t-X(t))\right)
    (\widetilde{W}^{\epsilon})'(x_1-\sigma t-X(t))\,\mathrm{d}x,\\
    X(0) &= 0.
\end{aligned}
\right.
\end{align}
In \eqref{eq:shiftfunction}, $\delta \in (0,1)$ is a sufficiently small constant and the specific constant $M$ is chosen as $M := {\|f_1''\|}_{L^\infty}+C\delta{\|B''\|}_{L^\infty}$, which will be used in the proof of Lemma \ref{lem:forW}. The existence of $X(t)$ follows from the a priori $L^\infty$-bound of $W^\eps$ obtained in Proposition \ref{lem:global_viscous} and $\partial_\xi\widetilde{W}^{\epsilon}, \partial_{\xi \xi}\widetilde{W}^{\epsilon}\in {L}^p$ for $1\leq p< \infty$. Since our proof is based on the relative entropy method up to a dynamical shift $X(t)$, we define the following relative entropy functional corresponding to a strictly convex entropy $E$.
\begin{align}\label{eq:relativeentropy}
    \mathcal{E}(W^\epsilon|\widetilde{W}^\epsilon) := {E}(W^\epsilon)-{E}(\widetilde{W}^\epsilon) - E'(\widetilde{W}^\epsilon)(W^\epsilon-\widetilde{W}^\epsilon).
\end{align}
As a consequence of the strictly convex entropy $E$, there exist constants $C_1, C_2>0$ depending only on ${\|\mathbf{U}_0^\eps\|}_{L^{\infty}(\Omega_T)}$, such that the following holds (see \cite{dafermos2005hyperbolic} for more details).
\begin{align}\label{eq:bound_on_E}
C_1(W^\epsilon-\widetilde{W}^\epsilon)^2\leq\mathcal{E}\left(W^\epsilon|\widetilde{W}^\epsilon\right)\leq C_2(W^\epsilon-\widetilde{W}^\epsilon)^2 .
\end{align}
We further define the corresponding relative entropy flux as $\mathcal{\mathbf{Q}} := (\mathcal{Q}_1,\ldots,\mathcal{Q}_d)$, where $\mathcal{Q}_i$'s are defined as
\begin{align}\label{eq:relativeentropyflux}
\mathcal{Q}_i(W^\epsilon|\widetilde{W}^\epsilon) = q_i(W^\epsilon)-q_i(\widetilde{W}^\epsilon)-E'(\widetilde{W}^\epsilon)(f_i(W^\epsilon)-f_i(\widetilde{W}^\epsilon)), \quad i= 1,\ldots, d.
\end{align}
In \eqref{eq:relativeentropyflux}, the entropy fluxes  $q_i$ satisfy $q_i' = E'f_i'$ for $i=1, 2,\ldots, d$. Henceforth, for notational simplicity we use $W^\epsilon_{\pm X}( \xi, x',t):= W^\epsilon( \xi\pm X(t), x',t)$.
\begin{lemma}\label{lem:forW}
    Consider the Cauchy problem  \eqref{eq:equationforW}-\eqref{eq:initialdataforW} with a given far-field state $W_-$ defined in \eqref{eq:initialdataforW} satisfying Assumption \ref{assumption2} and let $\widetilde{W}^\epsilon$ be the viscous shock satisfying \eqref{eq:viscouseqforW}. Then for any solution $W^\epsilon$ of \eqref{eq:equationforW}-\eqref{eq:initialdataforW} with initial data satisfying $W_0^\epsilon- \widetilde{W}^\epsilon\in (L^2\cap L^\infty)(\Omega_s)$, there exists an absolute continuous shift $X(t)$ such that the following estimate holds.
 \begin{align}\label{eq:contractionforW}
     \int_{\Omega_s} \left|W^\epsilon(x,t)-\widetilde{W}^\epsilon(x_1-\sigma t-X(t))\right|^2dx \leq \int_{\Omega_s}\left|W_0^\epsilon-\widetilde{W}^\epsilon\right|^2\,dx.
 \end{align}
 Moreover, suppose that $W_0^\epsilon-\widetilde{W}^\epsilon\in (L^1\cap L^\infty)(\Omega_s)$. Then we have the following time decay estimate for $W^\epsilon$ and $\dot{X}(t)$ for all $t\in (0,T]$.
 \begin{align}\label{eq:decayinW}
     {\|W^\epsilon(\cdot, t)-\widetilde{W}^\epsilon(\cdot-\sigma t-X(t))\|}_{L^2(\Omega_s)} = O(t^{-\frac{1}{4}}),
 \end{align}
 and 
 \begin{align}
     |\dot{X}(t)| = O(t^{-\frac{1}{4}}).
 \end{align}
\end{lemma}
\begin{proof}
   In view of \eqref{eq:relativeentropy}, we take the time derivative of the averaged relative entropy functional $\mathcal{E}$.  Using \eqref{eq:equationforW}, \eqref{eq:viscouseqforW}, \eqref{eq:relativeentropy}, and \eqref{eq:relativeentropyflux}, it yields
\begin{align}
        \dfrac{d}{dt}\int_{\Omega_s}&\mathcal{E}\left(W^\epsilon\,\big|\,\widetilde{W}_{-X}^\epsilon\right)\, dx'\, d\xi  \notag\\
&=\int_{\Omega_s}
\left(E'(W^\epsilon)-E'\left(\widetilde{W}_{-X}^{\eps}\right)\right)\partial_t W^\epsilon
-E''\left(\widetilde{W}_{-X}^{\eps}\right)
\left(W^\epsilon-\widetilde{W}_{-X}^{\eps}\right)
\partial_t\widetilde{W}_{-X}^{\eps}
\, dx'\, d\xi  \notag\\[0.5ex]
&=\int_{\Omega_s}
\left(E'(W)-E'\left(\widetilde{W}_{-X}^{\eps}\right)\right)
\left(\sigma W^\epsilon_\xi+\epsilon \Delta B(W^\epsilon)-\operatorname{div}F(W^\epsilon)\right)\, dx'\, d\xi \notag\\[0.5ex]
&\quad
-\int_{\Omega_s} E''\left(\widetilde{W}_{-X}^{\eps}\right)
\left(W^\epsilon-\widetilde{W}_{-X}^{\eps}\right)
\left(
-\dot{X}\,\partial_\xi\widetilde{W}_{-X}^{\eps}
+\sigma\partial_\xi\widetilde{W}_{-X}^{\eps}
+\epsilon \Delta B(\widetilde{W}_{-X}^{\eps})
-\operatorname{div}F\!\left(\widetilde{W}_{-X}^{\eps}\right)
\right)
\, dx'\, d\xi  \notag
\\[0.5ex]
&=\dot{X}\int_{\Omega_s}
E''\left(\widetilde{W}_{-X}^{\eps}\right)
\left(W^\epsilon-\widetilde{W}_{-X}^{\eps}\right)
\partial_\xi\widetilde{W}_{-X}^{\eps}
\, dx'\, d\xi  \notag
\\[0.5ex]
&\quad
+\int_{\Omega_s}
\sigma\Bigl(
\left(E'(W^\epsilon)-E'\!\left(\widetilde{W}_{-X}^{\eps}\right)\right)\partial_\xi W^\epsilon
-E''\left(\widetilde{W}_{-X}^{\eps}\right)
\left(W^\epsilon-\widetilde{W}_{-X}^{\eps}\right)
\partial_\xi\widetilde{W}_{-X}^{\eps}
\Bigr)
\, dx'\, d\xi  \notag
\\[0.5ex]
&\quad
-\int_{\Omega_s}
\operatorname{div}
\!\left(q\!\left(W^\epsilon|\widetilde{W}_{-X}^{\eps}\right)\right)
+E''\!\left(\widetilde{W}_{-X}^{\eps}\right)
\partial_\xi\widetilde{W}_{-X}^{\eps}
\,f_1(W^\epsilon\,\big|\,\widetilde{W}_{-X}^{\eps}) \, dx'\, d\xi \notag
\\[0.5ex]
&\quad
+\int_{\Omega_s} \sum_{i=2}^{n}
E''\!\left(\widetilde{W}_{-X}^{\eps}\right)
\partial_{x_i}\widetilde{W}_{-X}^{\eps}\,
f_i (W^\epsilon\,\big|\,\widetilde{W}_{-X}^{\eps})
\, dx'\, d\xi  \notag
\\[0.5ex]
&\quad
+\int_{\Omega_s}
\left(E'(W^\epsilon)-E'\!\left(\widetilde{W}_{-X}^{\eps}\right)\right)\epsilon \Delta B(W^\epsilon)
-E''\!\left(\widetilde{W}_{-X}^{\eps}\right)
\left(W^\epsilon-\widetilde{W}_{-X}^{\eps}\right)
\epsilon \Delta B(\widetilde{W}_{-X}^{\eps})
\, dx'\, d\xi . \label{eq:timegrowthofE}
\end{align}
Choosing $E(W^\epsilon) = \frac{(W^\epsilon)^2}{2}$  along with ${W^\epsilon|}_{\pm \infty}$ =   ${\widetilde{W}^\epsilon|}_{\pm \infty}$, \eqref{eq:timegrowthofE} reduces to 
\begin{align*}
\dfrac{\mathrm{d}}{\mathrm{d}t}&\int_{\Omega_s}
\frac{{|W^\epsilon-\widetilde{W}_{-X}^{\eps}|}^2}{2}\, dx'\, d\xi \\
&=
\dot{X}
\int_{\Omega_s}
\left(W^\epsilon-\widetilde{W}_{-X}^{\eps}\right)
\partial_{\xi}\widetilde{W}_{-X}^{\eps}
\, dx'\, d\xi 
+\sigma
\int_{\Omega_s}
\left(
\dfrac{{|W^\epsilon-\widetilde{W}_{-X}^{\eps}|}^2}{2}
\right)_{\xi}
\, dx'\, d\xi 
\\[0.5ex]
&
-
\int_{\Omega_s}
f_1 (W^\epsilon\,\big|\,\widetilde{W}_{-X}^{\eps})
\partial_{\xi}\widetilde{W}_{-X}^{\eps}
\, dx'\, d\xi 
+
\epsilon \int_{\Omega_s}
\left(W^\epsilon-\widetilde{W}_{-X}^{\eps}\right)
\Delta\!\left(B(W^\epsilon)-B(\widetilde{W}_{-X}^{\eps})\right)
\, dx'\, d\xi \\[0.5cm]
&  = \dot{X}
\int_{\Omega_s}
\left(W^\epsilon-\widetilde{W}_{-X}^{\eps}\right)
\partial_{\xi}\widetilde{W}_{-X}^{\eps}
\, dx'\, d\xi -
\int_{\Omega_s}
f_1 (W^\epsilon\,\big|\,\widetilde{W}_{-X}^{\eps})
\partial_{\xi}\widetilde{W}_{-X}^{\eps}
\, dx'\, d\xi \\[0.5cm]
& -\epsilon\int_{\Omega_s} \nabla \left(W^\epsilon -\widetilde{W}_{-X}^{\eps}\right). \nabla \left(B(W^\epsilon)-B(\widetilde{W}_{-X}^{\eps})\right)\, dx'\, d\xi \\[0.5cm]
& = \dot{X}
\int_{\Omega_s}
\left(W^\epsilon-\widetilde{W}_{-X}^{\eps}\right)
\partial_{\xi}\widetilde{W}_{-X}^{\eps}
\, dx'\, d\xi -
\int_{\Omega_s}
f_1 (W^\epsilon\,\big|\,\widetilde{W}_{-X}^{\eps})
\partial_{\xi}\widetilde{W}_{-X}^{\eps}
\, dx'\, d\xi \\[0.5cm]
& - \epsilon \int_{\Omega_s} B'(W^\epsilon)\left|\nabla(W^\epsilon-\widetilde{W}_{-X}^{\eps})\right|^2\, dx'\, d\xi - \epsilon \int_{\Omega_s} \left(B'(W^\epsilon)-B'(\widetilde{W}_{-X}^{\eps})\right)\nabla \widetilde{W}_{-X}^{\eps}.\nabla\left(W^\epsilon-\widetilde{W}_{-X}^{\eps}\right)\, dx'\, d\xi \\[0.5cm]
& = \dot{X}
\int_{{\Omega_s}}
\left(W^\epsilon-\widetilde{W}_{-X}^{\eps}\right)
\partial_{\xi}\widetilde{W}_{-X}^{\eps}
\, dx'\, d\xi -
\int_{{\Omega_s}}
f_1 (W^\epsilon\,\big|\,\widetilde{W}_{-X}^{\eps})
\partial_{\xi}\widetilde{W}_{-X}^{\eps}
\, dx'\, d\xi \\[0.5cm]
& - \epsilon \int_{\Omega_s} B'(W^\epsilon)\left|\nabla(W^\epsilon-\widetilde{W}_{-X}^{\eps})\right|^2\, dx'\, d\xi - \epsilon \int_{\Omega_s} \left(B'(W^\epsilon)-B'(\widetilde{W}_{-X}^{\eps})\right)\partial_\xi \widetilde{W}_{-X}^{\eps}\partial_\xi\left(W^\epsilon-\widetilde{W}_{-X}^{\eps}\right)\, dx'\, d\xi .
\end{align*}
After a change of variable $\xi \longrightarrow \xi +X(t)$, we obtain
 \begin{align}\label{eq:entropyestimate}
\dfrac{\mathrm{d}}{\mathrm{d}t}\int_{{\Omega_s}}&
\dfrac{\left|{W}_X^\epsilon-{\widetilde{W}}^\epsilon\right|^2}{2}\, dx'\, d\xi \nonumber\\
& \hspace{0.5cm}= \dot{X}
\int_{{\Omega_s}}
\left(W^\epsilon_X-\widetilde{W}^\epsilon\right)
\partial_{\xi}\widetilde{W}^\epsilon
\, dx'\, d\xi -
\int_{{\Omega_s}}
f_1 (W^\epsilon_X\,\big|\,\widetilde{W}^\epsilon)
\partial_{\xi}\widetilde{W}^\epsilon
\, dx'\, d\xi \nonumber\\[0.4cm]
&\hspace{0.5cm}- \epsilon \int_{\Omega_s} \left(B'(W^\epsilon_X)-B'(\widetilde{W}^\epsilon)\right)\partial_\xi\left(W^\epsilon_X-\widetilde{W}^\epsilon\right)\partial_\xi \widetilde{W}^\epsilon\, dx'\, d\xi -\epsilon \int_{\Omega_s} B'(W^\epsilon_X)\left|\nabla(W^\epsilon_X-\widetilde{W}^\epsilon)\right|^2\, dx'\, d\xi \nonumber\\[0.2cm]
&\hspace{0.5cm}:= \dot{X}\mathcal{Y}(W^\epsilon_X)+ \mathcal{B}(W^\epsilon_X)-\mathcal{G}(W^\epsilon_X),
\end{align}
where
\begin{align}
    &\mathcal{Y}(W^\epsilon_X) =\int_{\Omega_s} \left(W^\epsilon_X-\widetilde{W}^\epsilon\right)
\partial_{\xi}\widetilde{W}^\epsilon\, dx'\, d\xi \label{eq:forY},\\
&\mathcal{B}(W^\epsilon_X) = -\int_{\Omega_s} f_1(W^\epsilon_X|\widetilde{W}^\epsilon)\partial_\xi\widetilde{W}^\epsilon \, dx'\, d\xi -  \epsilon\int_{\Omega_s} \left(B'(W^\epsilon_X)-B'(\widetilde{W}^\epsilon)\right)\partial_\xi\left(W^\epsilon_X-\widetilde{W}^\epsilon\right)\partial_\xi \widetilde{W}^\epsilon \, dx'\, d\xi \label{eq:forB},\\
& \mathcal{G}(W^\epsilon_X) = \epsilon\int_{\Omega_s} B'(W_X^\epsilon)\left|\nabla(W^\epsilon_X-\widetilde{W}^\epsilon)\right|^2\, dx'\, d\xi .\label{eq:forG}
\end{align}
Now we rewrite the functionals \eqref{eq:forY}, \eqref{eq:forB}, \eqref{eq:forG} and $\dot{X}(t)$ in terms of the following variables similar to \cite{MR3953019},
\begin{align}\label{eq:transformation}
    z = \dfrac{1}{\delta}(W_--\widetilde{W}^\epsilon),\quad \mathcal{W} = (W^\epsilon_X-\widetilde{W}^\epsilon)o(z^{-1},\mathbf{I}).
\end{align}
In \eqref{eq:transformation}, $o$ is the composition operator and $\mathbf{I}$ is the identity operator on $\mathbb{T}^{d-1}$ to $\mathbb{T}^{d-1}$ . Note that
\begin{align}\label{eq:transformation2}
    \dfrac{dz}{d\xi} = -\dfrac{\partial_\xi\widetilde{W}^\epsilon}{\delta}.
\end{align}
Using the transformation \eqref{eq:transformation}, \eqref{eq:transformation2},  we rewrite $\dot{X}(t)$ and \eqref{eq:forY},  as follows
\begin{align}\label{eq:estimateY}
    \dot{X}(t)= M\int_{\mathbb{T}^{d-1}}\overline{\mathcal{W}}\,dx', \qquad \mathcal{Y}(W^\epsilon_X) = -\delta\int_{\mathbb{T}^{d-1}} \overline{\mathcal{W}}\, dx',
\end{align}
where $\overline{\mathcal{W}} = \int_0^1 \mathcal{W}\, dz$. 

For the bad terms $\mathcal{B}$ in \eqref{eq:entropyestimate}, we rewrite $\mathcal{B} := \mathcal{B}_1+\mathcal{B}_2$ with
\[
\mathcal{B}_1(W^\epsilon_X)
= -\int_{\Omega_s} f_1(W^\epsilon_X|\widetilde{W}^\epsilon)
\partial_\xi\widetilde{W}^\epsilon \, dx'\, d\xi
\]
and 
\[
\mathcal{B}_2(W^\epsilon_X)
= - \epsilon \int_{\Omega_s} \left(B'(W^\epsilon_X)-B'(\widetilde{W}^\epsilon)\right)
\partial_\xi\left(W^\epsilon_X-\widetilde{W}^\epsilon\right)
\partial_\xi \widetilde{W}^\epsilon \, dx'\, d\xi.
\]
In view of strict convexity of $f_1$, by \eqref{eq:bound_on_E}, \eqref{eq:transformation} and \eqref{eq:transformation2}, we estimate $\mathcal{B}_1$ as follows
\begin{align}
\mathcal{B}_1(W^\epsilon_X)
&= -\int_{\Omega_s} f_1(W^\epsilon_X|\widetilde{W}^\epsilon)
\partial_\xi\widetilde{W}^\epsilon \, dx'\, d\xi \notag\\
&\leq \delta{\|f''_1\|}_{L^{\infty}}
\int_{\mathbb{T}^{d-1}}\int_0^1 \mathcal{W}^2\, dx'\,dz.
\label{eq:estimateB_1}
\end{align}
Now for $\mathcal{B}_2$,  we proceed as follows
\begin{align}
\mathcal{B}_2(W^\epsilon_X)
&= - \epsilon \int_{\Omega_s} \left(B'(W^\epsilon_X)-B'(\widetilde{W}^\epsilon)\right)
\partial_\xi\left(W^\epsilon_X-\widetilde{W}^\epsilon\right)
\partial_\xi \widetilde{W}^\epsilon \, dx'\, d\xi 
\notag\\
&=  - \epsilon \int_{\Omega_s} B''(W^*)
\left(W^\epsilon_X-\widetilde{W}^\epsilon\right)
\partial_\xi\left(W^\epsilon_X-\widetilde{W}^\epsilon\right)
\partial_\xi \widetilde{W}^\epsilon \, dx'\, d\xi 
\notag\\
&\leq {\|B''\|}_{L^\infty}\delta
\int_{\mathbb{T}^{d-1}}\int_0^1
\mathcal{W}\mathcal{W}_z\,\dfrac{dz}{d\xi}\, dx'\, dz,
\label{eq:estimateB_21}
\end{align}
where $W^*$ lies between $W^\epsilon_X$ and $\widetilde{W}^\epsilon$.  

Using the properties of viscous shock (see \cite{MR4195742} for more details), we must have $|\partial_\xi\widetilde{W}^\epsilon|\leq C\delta^2$. Thus, it follows that 
\begin{align}\label{eq:fordz}
    \left|\dfrac{dz}{d\xi}\right| = \left|-\dfrac{\partial_\xi\widetilde{W}^\epsilon}{\delta}\right| \leq C\delta.
\end{align}
Therefore, using Young's inequality and \eqref{eq:fordz}, we further estimate \eqref{eq:estimateB_21} as
\begin{align}\label{eq:estimateB_2}
    \mathcal{B}_2(W^\epsilon_X) \leq C\delta^2{\|B''\|}_{L^\infty}\int_{\mathbb{T}^{d-1}}\int_0^1 \mathcal{W}^2\, dx'\,dz + \dfrac{\delta}{2}{\|B''\|}_{L^\infty}\int_{\mathbb{T}^{d-1}}\int_0^1 |\mathcal{W}_z|^2\,\dfrac{dz}{d\xi}\,dx'\,dz.
\end{align}
Before estimating the good term $\mathcal{G}$ in \eqref{eq:entropyestimate}, as defined in \eqref{eq:forG}, we recall the following result from \cite{MR4188324} with a minor modification. For a strictly smooth convex function $f_1$, there exist positive constants $C$ and $\delta_0\in (0,1)$ such that, for any $0<\delta<\delta_0$,  and $z\in[0,1]$, we have the following estimate.
\begin{align}\label{eq:estimateofdzdxi} 
\left|
\dfrac{1}{z(1-z)}\dfrac{dz}{d\xi}
-\dfrac{\delta}{2 \epsilon B'(W)}f_1''(W_-)
\right|
\le \dfrac{C}{\epsilon}\delta^2.
\end{align}
Using \eqref{eq:estimateofdzdxi} with the a priori bound  of $B'$ (see \eqref{eq:Bprime_bounds}), i.e.,  $km_0\leq B'\leq kM_0$, we estimate $\mathcal{G}(W^\epsilon_X)$ in \eqref{eq:entropyestimate} as follows
\begin{align}\label{eq:estimateG}
        \mathcal{G}(W^\epsilon_X) &= \epsilon\int_{\Omega_s} B'(W^\epsilon)\left|\nabla(W^\epsilon_X-\widetilde{W})\right|^2\, dx'\, d\xi \nonumber\\
        &= \epsilon\int_{\mathbb{T}^{d-1}}\int_\mathbb{R} B'(W)\left|\partial_\xi(W^\epsilon_X-\widetilde{W})\right|^2\, dx'\, d\xi + \sum_{i=2}^d\epsilon\int_{\mathbb{T}^{d-1}}\int_\mathbb{R} B'(W^\epsilon)\left|\partial_{x_i}(W^\epsilon_X-\widetilde{W})\right|^2\, dx'\, d\xi \nonumber\\
        & \geq k\epsilon m_0 \int_{\mathbb{T}^{d-1}}\int_0^1\left|\partial_z\mathcal{W}\right|^2\dfrac{dz}{d\xi}\,dx'\,dz+\sum_{i=2}^d k\epsilon m_0\int_{\mathbb{T}^{d-1}}\int_\mathbb{R} \left|\partial_{x_i}(\mathcal{W})\right|^2\,\dfrac{d\xi}{dz}\,dx'\,dz \nonumber\\
        & \geq k\epsilon m_0\left(\dfrac{\delta f_1''(W_-)}{2 \epsilon kM_0}-\dfrac{C}{\epsilon}\delta^2\right)\int_{\mathbb{T}^{d-1}}\int_0^1z(1-z)\left|\partial_z\mathcal{W}\right|^2\,dx'\,dz\nonumber\\
        &\hspace{2.9cm} +\sum_{i=2}^d k\epsilon m_0\left(\dfrac{\delta f_1''(W_-)}{2 \epsilon km_0}+\dfrac{C}{\epsilon}\delta^2\right)^{-1}\int_{\mathbb{T}^{d-1}}\int_0^1\dfrac{\left|\partial_{x_i}(\mathcal{W})\right|^2}{z(1-z)}\,dx'\,dz.
    \end{align}
Now using \eqref{eq:estimateY}, \eqref{eq:estimateB_1}, \eqref{eq:estimateB_2}, \eqref{eq:estimateofdzdxi}, and \eqref{eq:estimateG}  in \eqref{eq:entropyestimate}, we have the following inequality 
\begin{align}\label{eq:estimateofall1}
        &\dfrac{1}{(\|f''_1\|_{L^\infty}+C\delta\|B''\|_{L^\infty})}\dfrac{1}{\delta}\left(\dot{X}\mathcal{Y}(W^\epsilon_X)+B(W^\epsilon_X)-\mathcal{G}(W^\epsilon_X)\right) \nonumber\\[1ex]
        &\leq -\left(\int_{\mathbb{T}^{d-1}}\overline{\mathcal{W}}dx'\right)^2+ \int_{\mathbb{T}^{d-1}}\int_0^1 \mathcal{W}^2\, dx'\,dz\nonumber\\[2ex]
        & \quad+ \dfrac{1}{{\|f''_1\|}_{L^\infty}+C\delta}\left(\left(\dfrac{\delta{\|B''\|}_{L^\infty} f_1''(W_-)}{4\epsilon km_0}+\dfrac{C}{\epsilon}\delta^2\right)- \left(\dfrac{  m_0f_1''(W_-)}{2 M_0}-C\delta\right)\right)\int_{\mathbb{T}^{d-1}}\int_0^1z(1-z)\left|\partial_z\mathcal{W}\right|^2\,dx'\,dz\nonumber\\[0.5cm]
        &\quad-\dfrac{k\epsilon m_0}{\delta^2({\|f''_1\|}_{L^\infty}+C\delta)} \left(\dfrac{ f_1''(W_-)}{2\epsilon km_0}+\dfrac{C}{\epsilon}\delta\right)^{-1}\,\,\sum_{i=2}^{d-1}\int_{\mathbb{T}^{d-1}}\int_0^1 \dfrac{\left|\partial_{x_i} \mathcal{W}\right|^2}{z(1-z)}\,dx'\,dz\nonumber\\[0.4cm]
        & = -\left(\int_{\mathbb{T}^{d-1}}\overline{\mathcal{W}}\,dx'\right)^2+ \int_{\mathbb{T}^{d-1}} \overline{\mathcal{W}}^2\,dx'\,dz + \int_{\mathbb{T}^{d-1}}\left(\int_0^1 \mathcal{W}^2\, dz-\overline{\mathcal{W}}^2-\int_0^1\dfrac{1}{2}z(1-z)\left|\partial_z \mathcal{W}\right|^2\,dz\right)\,dx'\nonumber\\[0.3cm]
        & \quad -\dfrac{k\epsilon m_0}{\delta^2({\|f''_1\|}_{L^\infty}+C\delta)} \left(\dfrac{ f_1''(W_-)}{2\epsilon km_0}+\dfrac{C}{\epsilon}\delta\right)^{-1}\,\,\sum_{i=2}^{d-1}\int_{\mathbb{T}^{d-1}}\int_0^1 \dfrac{\left|\partial_{x_i} \mathcal{W}\right|^2}{z(1-z)}\,dx'\,dz\nonumber\\[0.3cm]
        &\quad+ \left(\dfrac{1}{2} + \dfrac{\left(\dfrac{\delta{\|B''\|}_{L^\infty} f_1''(W_-)}{4\epsilon km_0}+\dfrac{C}{\epsilon}\delta^2\right)- \left(\dfrac{  m_0f_1''(W_-)}{2 M_0}-C\delta\right)}{{\|f''_1\|}_{L^\infty}+C\delta}\right)\int_{\mathbb{T}^{d-1}}\int_0^1z(1-z)\left|\partial_z\mathcal{W}\right|^2\,dx'\,dz.
    \end{align}
Using the Poincaré-type inequality \eqref{eq: poincare_type_inequality}, we have
\begin{align*}
    \int_0^1\mathcal{W}^2\,dz-\overline{\mathcal{W}}^2-\int_0^1\dfrac{1}{2}z(1-z)\left|\partial_z \mathcal{W}\right|^2\,dz \leq 0,
\end{align*}
Now \eqref{eq:estimateofall1} becomes
    \begin{align}
        &\dfrac{1}{({\|f''_1\|}_{L^\infty}+C\delta{\|B''\|}_{L^\infty})}\dfrac{1}{\delta}\left(\dot{X}\mathcal{Y}(W^\epsilon_X)+\mathcal{B}(W^\epsilon_X)-\mathcal{G}(W^\epsilon_X)\right) \nonumber\\[1ex]
        & \leq  \underbrace{-\left(\int_{\mathbb{T}^{d-1}}\overline{\mathcal{W}}\,dx'\right)^2+ \int_{\mathbb{T}^{d-1}} \overline{\mathcal{W}}^2\,dx'\,dz- \dfrac{k\epsilon m_0\left(\dfrac{ f_1''(W_-)}{2\epsilon km_0}+\dfrac{C}{\epsilon}\delta\right)^{-1}}{\delta^2(\|f''_1\|_{L^\infty}+C\delta)}\sum_{i=2}^{d-1}\int_{\mathbb{T}^{d-1}}\int_0^1 \dfrac{\left|\partial_{x_i} \mathcal{W}\right|^2}{z(1-z)}\,dx'\,dz}_{R_1} \nonumber\\[1ex]
        & \underbrace{+ \left(\dfrac{1}{2} + \dfrac{\left(\dfrac{\delta{\|B''\|}_{L^\infty} f_1''(W_-)}{4\epsilon km_0}+\dfrac{C}{\epsilon}\delta^2\right)- \left(\dfrac{  m_0f_1''(W_-)}{2 M_0}-C\delta\right)}{{\|f''_1\|}_{L^\infty}+C\delta}\right)}_{R_{2}}\int_{\mathbb{T}^{d-1}}\int_0^1z(1-z)\left|\partial_z \mathcal{W}\right|^2\,dx'\,dz.
    \end{align}
Using the fact that $|z(1-z)|\leq \frac{1}{4}$ on $|z|\leq1$ , and the Poincaré inequality \cite{Evans} on $\mathbb{T}^{d-1}$, we estimate $R_1$ as follows

    \begin{align}
        R_1 &\leq -\left(\int_{\mathbb{T}^{d-1}}\overline{\mathcal{W}}\,dx'\right)^2+ \int_{\mathbb{T}^{d-1}} \overline{\mathcal{W}}^2\,dx'\,dz-\dfrac{4k\epsilon m_0\left(\dfrac{ f_1''(W_-)}{2\epsilon k m_0}+\dfrac{C}{\epsilon}\delta\right)^{-1}}{\delta^2({\|f''_1\|}_{L^\infty}+C\delta)}\sum_{i=2}^{d-1}\int_{\mathbb{T}^{d-1}}\int_0^1 \left|\partial_{x_i} \mathcal{W}\right|^2\,dx'\,dz\nonumber\\[2ex]
        & \leq \dfrac{1}{4\pi^2}\sum_{i=2}^{d-1}\int_{\mathbb{T}^{d-1}} \left|\partial_{x_i}\overline{\mathcal{W}}\right|^2\,dx'- \dfrac{4k\epsilon m_0\left(\dfrac{ f_1''(W_-)}{2\epsilon km_0}+\dfrac{C}{\epsilon}\delta\right)^{-1}}{\delta^2(\|f''_1\|_{L^\infty}+C\delta)}\sum_{i=2}^{d-1}\int_{\mathbb{T}^{d-1}}\int_0^1 \left|\partial_{x_i} \mathcal{W}\right|^2\,dx'\,dz\nonumber\\[1ex]
        & \leq \left(\dfrac{1}{4\pi^2}-\dfrac{4k\epsilon m_0\left(\dfrac{ f_1''(W_-)}{2\epsilon km_0}+\dfrac{C}{\epsilon}\delta\right)^{-1}}{\delta^2(\|f''_1\|_{L^\infty}+C\delta)}\right)\sum_{i=2}^{d-1}\int_{\mathbb{T}^{d-1}}\int_0^1 \left|\partial_{x_i} \mathcal{W}\right|^2\,dx'\,dz.
    \end{align}
Since $m_0>0$ , and $f_1''(W_-)>0$, by choosing $\delta_0=\delta_1$ sufficiently small, we obtain $R_1\leq0$. Moreover, $R_2\leq0$ if and only if
\begin{align}\label{eq:R2r}
   R_{2}:= \dfrac{C}{\epsilon}\delta^2+\left(\dfrac{\|B''\|_{L^\infty}f_1''(W_-)}{4\epsilon km_0}+\dfrac{3C}{2}\right)\delta + \dfrac{1}{2}\left({\|f_1''\|}_{L^\infty}- \dfrac{m_0}{M_0}f_1''(W_-)\right)\leq0.
\end{align}
Thus, by  Assumption \ref{assumption2}, and choosing sufficiently small $\delta_0=\delta_2$ for a fixed $\epsilon$,  \eqref{eq:R2r} follows immediately. Hence, we get the contraction for $W^\epsilon$ in $L^2(\Omega_s)$ by choosing $\delta_0=\min\{\delta_1, \delta_2\}$. For the decay estimate \eqref{eq:decayinW}, the proof is almost identical to that of Theorem 1.1 in  \cite{MR4876608}. Thus, we omit the details here. Now it remains to show the decay estimate of the function $\dot{X}(t)$. Observe that using \eqref{eq:shiftfunction} , and \eqref{eq:decayinW} , and applying  Cauchy-Schwarz inequality, one can obtain the estimate
\begin{align*}
    |\dot{X}(t)|&\leq \dfrac{M}{\delta}{\|W^\epsilon-\widetilde{W}_{-X}^\epsilon\|}_{L^2({\Omega_s})}{\|\partial_\xi\widetilde{W}^\epsilon\|}_{L^2(\mathbb{R})}\\
    & = O(t^{-\frac{1}{4}}),
\end{align*}
which completes the proof of Lemma \ref{lem:forW}.
\end{proof}
\begin{remark}\label{rem:forviscousshock}
   It is noteworthy to highlight that the results obtained in Lemma \ref{lem:existenceofviscouswave}, and  \ref{lem:forW}   are of independent interest beyond their role in the analysis of the full system \eqref{eq:homogeneous_KK_viscos_2}. In particular, the existence of the viscous shock profile and the time-asymptotic decay of the perturbation $(W^\epsilon-\widetilde{W}^\epsilon)$ provide a self-contained stability result for the viscous scalar conservation laws of the form \eqref{eq:equationforW}-\eqref{eq:initialdataforW} with nonlinear viscosity. This extends the existing results for the scalar viscous conservation laws with linear viscosity \cite{MR4876608, MR3953019}.
\end{remark}
\subsection{Proof of Theorem \ref{main_theorem_2}}
In this Section, we establish the $L^2$-decay of the perturbation  $(u_i^\epsilon-\widetilde{u}_i^\epsilon), \,i=1, 2, \ldots, n$ by exploiting the corresponding $L^2$-decay estimate for $(W^\epsilon-\widetilde{W}^\epsilon)$ obtained in Lemma \ref{lem:forW}. Recall that the initial data $\mathbf{U}_0^\eps\in \widetilde{\mathcal U}$. Moreover, the variables $\theta_i^\eps$, $i=2,\ldots,n$, satisfy the transport equations \eqref{eq:ratio_transport_equation_main}. Since the initial data for the ratios $\theta_i^{\epsilon}:=u^\epsilon_i/u_1^\eps=\gamma_i, \, i=2, 3, \ldots, n$ are constants, it follows that $\theta_i^\epsilon$ remain constant along the viscous flow for all time $t\in (0,T]$. Consequently,
\begin{align}\label{eq:fortheta}
\theta^\epsilon_i(x,t)\equiv \gamma_i, \qquad i=2,\ldots,n, \quad \forall\, (x,t)\in \Omega_s\times(0,T].
\end{align}
Since $g$ is a continuously differentiable function of $\bm{\theta}$, and $\bm{\theta}$ remains bounded, $g(\bm{\theta})$ is bounded. From \eqref{eq:ratio_transport_equation_main}, we further deduce for the viscous profile $\widetilde{\theta}_i^\epsilon(\xi)$ that $\partial_\xi \widetilde{\theta}_i^\epsilon = 0$. Therefore, $\widetilde{\theta}_i^\epsilon$ remains constant, i.e.,
\begin{align}\label{eq:viscoustheta}
\widetilde{\theta}_i^\epsilon(\xi) = \gamma_i, \qquad i=2,\ldots,n, \quad \forall\, \xi \in \mathbb{R}
\end{align}
Now, using the reconstruction formula \eqref{eq:reconstruction_formula}, and the fact that $g$ is bounded, we deduce that
\begin{align}\label{eq:differenceinu_1}
    \left|u_1^\epsilon - \widetilde{u}_1^\epsilon\right| \leq  \dfrac{1}{|g|}\left|W^\epsilon - \widetilde{W}^\epsilon\right|
\end{align}
Therefore, using Lemma \ref{lem:forW}, and \eqref{eq:differenceinu_1}, we have the $L^2$ decay estimate for $(u_1^\epsilon-\widetilde{u}_1^\epsilon)$ as
\begin{equation}\label{eq:decayinu_1}
    \begin{aligned}
     {\|u_1^\epsilon(\cdot, t)-\widetilde{u}_1^\epsilon(\cdot-\sigma t-X(t))\|}_{L^2(\Omega_s)} &\leq C {\|W^\epsilon(\cdot, t)-\widetilde{W}^\epsilon(\cdot-\sigma t-X(t))\|}_{L^2(\Omega_s)} \\
     & = O(t^{-\frac{1}{4}})
 \end{aligned}
\end{equation}
Now it remains to establish the decay estimate for $u_i^\epsilon-\widetilde{u}_i^\epsilon, \, i=2, 3, \ldots, n$. For this, using \eqref{eq:fortheta}, \eqref{eq:viscoustheta} and  \eqref{eq:decayinu_1}, we obtain
\begin{equation}
    \begin{aligned}
    {\|u_i^\epsilon(\cdot, t) - \widetilde{u}_i^\epsilon(\cdot-\sigma t-X(t))\|}_{L^2(\Omega_s)} &= \gamma_i {\|u_1^\epsilon(\cdot, t) - \widetilde{u}_1^\epsilon(\cdot-\sigma t-X(t)))\|}_{L^2(\Omega_s)}, \quad i = 2,\ldots,n\\
    &= O(t^{-\frac{1}{4}}).
\end{aligned}
\end{equation}
This completes the proof of Theorem \ref{main_theorem_2}.
\section{Conclusions and future outlook}\label{sec: conclusions}
We established the global well-posedness of weak solutions to the Cauchy problem for a wide class of $n\times n$ multi-dimensional hyperbolic systems. We achieved this by introducing a novel viscous approximation, which enabled the use of a compactness framework to extract the strong limit of the approximating sequence. Through this approximation, we transformed the viscous system into a single scalar equation for the nonlinear variable $V^\eps$, coupled with $(n-1)$ transport equations for $\theta_i^\eps$. This unique structure allowed us to derive uniform a priori estimates independent of the diffusion coefficient $\epsilon$. By utilizing these estimates, we proved the global-in-time existence and uniqueness of weak solutions for the inviscid system.

Furthermore, we proved that the proposed viscous approximation admits a viscous shock profile. Using the relative entropy method, we analyzed the $L^2$ time-decay behavior of viscous shock waves associated with the viscous system. In particular, under suitable assumptions on the initial data and an additional structural condition on the state space, we proved the asymptotic $L^2$ time-decay of the viscous shock waves of the approximation.

An important outcome of this work is the construction of a positively invariant domain for a broad class of hyperbolic systems, which enables us to design positivity-preserving numerical schemes, especially for thin film applications. Looking forward, it would be interesting to determine whether uniqueness for the constructed weak solutions holds for less regular initial data. A natural extension is to generalize our approach to systems where the flux differs in each spatial direction. Additionally, we aim to make use of the special structure of the viscous system to investigate the stability and time-asymptotic behavior of composite wave patterns, including viscous contact waves.
\bibliographystyle{abbrv}
\bibliography{citation}
\end{document}